\documentclass{article}
\usepackage[affil-it]{authblk}
\usepackage{pdfsync}
\usepackage{amsmath,amssymb}
\usepackage{enumerate}
\usepackage{xspace}
\usepackage{pgfplots}
\usetikzlibrary{math}

\usepackage{boxedminipage}
\usepackage{algorithmicx}
\usepackage[ruled]{algorithm}
\usepackage{listings}
\usepackage{tabularx}
\usepackage{fullpage}
\usepackage[colorlinks,citecolor=cyan,linkcolor=magenta]{hyperref}
\usepackage{amsmath,amsthm,amssymb,amsfonts, bm, bbm, stmaryrd} 
\usepackage{graphics} 
\usepackage{subcaption} 
\usepackage{enumitem}
\usepackage{tikz-cd}
\usepackage{url}

\usepackage{mathtools}
\usepackage{tabularx}
\usepackage{array}
\usepackage{booktabs}
\usepackage{IEEEtrantools} 
\usepackage{mathabx}
\usepackage{xcolor}

\usepackage[backend=biber,style=numeric,giveninits, maxbibnames=9]{biblatex} 
\usepackage{setspace}
\renewbibmacro{in:}{%
  \ifentrytype{article}{}{\printtext{\bibstring{in}\intitlepunct}}}
\numberwithin{equation}{section}

\makeatletter
\def\BState{\State\hskip-\ALG@thistlm}
\makeatother

\graphicspath{ {images/} }

\usepackage{l4_macro}

\usepackage{mathrsfs}
\provideboolean{usemathrsfs}
\setboolean{usemathrsfs}{true}

\usepackage{hyperref}

\title{Duality based error control for the Signorini problem}

\author{Ben S. Ashby}
\affil{School of Mathematical \& Computer Sciences,\\ Heriot-Watt University, Edinburgh, UK.}

\author{Tristan Pryer}
\affil{Department of Mathematical Sciences, \\ University of Bath, Bath, UK.}

\date{}

\begin{document}

\maketitle

\begin{abstract}
  In this paper we study the a posteriori bounds for a conforming
  piecewise linear finite element approximation of the Signorini
  problem.  We prove new rigorous a posteriori estimates of residual
  type in $\leb{p}$, for $p \in (4,\infty)$ in two spatial
    dimensions. This new analysis treats
  the positive and negative parts of the discretisation error
  separately, requiring a novel sign- and bound-preserving
  interpolant, which is shown to have optimal approximation
  properties.  The estimates rely on the sharp dual stability results
  on the problem in $\sob{2}{(4 - \varepsilon)\slash 3}$
  for any $\varepsilon \ll 1$.  We summarise extensive
  numerical experiments aimed at testing the robustness of the
  estimator to validate the theory.
\end{abstract}

\section{Introduction}

{
Consider the open and convex set $\W \subset \reals^2$ with a
polygonal boundary $\partial \W$ and let $\vec n$ denote the outward
pointing normal to $\partial\W$. This paper focuses on a variational
inequality that emerges from the study of the following partial
differential equation (PDE):
\begin{equation}
  - \Delta u + u
  =
  f
  \text{ in } \W,  \label{eq:diff_eq}
\end{equation}
accompanied by the boundary conditions
\begin{equation}
  u
  \geq
  0, \quad
  \nabla u \cdot \vec n
  \geq
  0, \quad
  u
  \nabla u \cdot \vec n
  =
  0
  \text{ on } \partial \W. \label{eq:constraints}
\end{equation}
The specified conditions dictate that, for any segment $I \subset
\partial \W$ where $u>0$, it follows that $\nabla u \cdot \vec n = 0$
over $I$. This problem, along with its variants, finds applications in
diverse areas such as elasto-plasticity, fluid flow in porous media
\cite{brezzi1976finite}, and finance \cite{zhang2007adaptive}. Its
transformation into a variational inequality has facilitated extensive
analysis, notably in \cite{lions1967variational}, which established
the existence and uniqueness of solutions across a broad class of
problems.}

{Significant progress has been made in deriving a priori error
estimates for finite element approximations to variational
inequalities, achieving optimal order in energy norms
\cite{brezzi1977error, falk1974error,
  MoscoStrang:1974,scarpini1977error}, with notable advancements also
reported in \cite{hild2012improved} and \cite{nitsche1977convergence}
for estimates in $\leb{\infty}(\W)$. Despite these developments,
results in norms other than the energy norm remain relatively
rare. For instance, \cite{mosco1977error} offers an a priori bound in
$\leb{2}(\W)$, albeit in a single spatial dimension. The literature
also documents bounds on trace norms along the Signorini boundary
\cite{steinbach_trace}, with more recent studies, such as
\cite{christof1754finite}, extending bounds to several low-order
global norms.}

{The theory of a posteriori error estimation for elliptic variational
inequalities is significantly less well-developed than that for
elliptic boundary value problems. Nonetheless, there have been several
works on the Signorini problem as well as the related problem where
the obstacle is interior to the domain rather than on the boundary. In
\cite{hoppe1994adaptive}, hierarchical estimates are used under the
assumption that quadratic finite elements give better approximation
than linears. Error estimates for variational inequalities with
nontrivial boundary data are given in \cite{krause2015efficient},
estimates for an alternative form using Lagrange multipliers are given
in \cite{braess2005posteriori, weiss2010posteriori}, with results for
the discontinuous Galerkin method given in \cite{walloth2019reliable}.}

Rigorous a posteriori error bounds of residual type were derived for
the variational inequality of the first kind for the Signorini problem
in \cite{ChenObstacle2000} for piecewise linear finite elements. Here,
the authors give an upper bound in $\sobh{1}(\W)$. Several other works
have considered a posteriori control in energy norms, see
\cite{hild2005posteriori}.  Pointwise error control was established in
\cite{NochettoObstacle2003}. Recent regularity results for the problem
in which constraints only apply at the boundary
\cite{christof1754finite} show that better rates can be achieved in
this case, paving the way for a posteriori error estimates in lower
order norms.

We remark that error estimates have been derived for this problem
using the dual-weighted residual methodology (see for example
\cite{suttmeier2008numerical}) to estimate the error in a target
quantity.  In particular, estimates follow for the
$\leb{2}(\W)$-error, however the necessary stability of the dual
problem enters as an assumption.

A priori and a posteriori error estimates in low order norms are
established using duality arguments, and the resulting error bounds in
for example $\leb{2}(\W)$ are standard in the literature for boundary
value problems. The Aubin-Nitsche duality argument allows a priori and
a posteriori bounds in non-energy norms, see
\cite{ainsworth2011posteriori} and the references therein.
Application of this methodology has proved to be a challenge in the
case of variational inequalities since the dual problem can lack the
necessary regularity. Indeed, given smooth data, solutions of the
Signorini problem are smooth away from the boundary, but suffer from
singular points at the boundaries of the contact set, that is, the set
of points at which the boundary changes type. In this work, using the
ideas of an a priori argument recently given in
\cite{christof1754finite}, the approach taken in \cite{mosco1977error}
in which different dual problems are considered for the positive and
negative parts of the error is adapted for the a posteriori setting
and combined with a two-sided approximation that takes the place of
the usual Lagrange interpolant. The novelty in this work is that we
are able to prove rigorous a posteriori error control in
$\leb{p}(\W)$, for $p \in (4,\infty)$.  The results in
\cite{christof1754finite} suggest that the exponent $4$ is a critical
value in the sense that piecewise linear finite element approximations
achieve optimal a priori orders of convergence in
$\sob{1}{4-\varepsilon}(\W)$ for any $\varepsilon \ll 1$ but become
increasingly suboptimal in $\sob{1}{p}(\W)$ for $p>4$. Increasingly
suboptimal convergence is also observed in $\leb{p}(\W)$ for $p>4$. We
observe the same critical exponent for our a posteriori bounds,
however the assumptions we require to derive them are weaker than the
a priori setting. Furthermore, we are able to regain optimal
convergence rates through adaptive refinement techniques.

{The main aim of this paper to understand the theoretical
  framework in which duality arguments can be applied to obtain a
  posteriori error bounds in non-energy norms. To that end, we give a
  brief discussion of the underlying assumptions, and those used
  elsewhere in the literature. We defer technical discussions to the
  main body of the article, but give a brief overview here to put our
  results in context.
\begin{enumerate}
\item 
  A regularity assumption on the structure of the contact set of the
  solution to the Signorini problem (known as
  \hyperref[struct_of_contact_set]{Condition $(\mathbf{A})$}) is
  required to prove regularity of the dual problem, and indeed such an
  assumption is commonplace in the literature
  \cite{belgacem1999extension,belhachmi2003quadratic,
    ben2000numerical, christof1754finite,
    haslinger1996numerical,hueber2005optimal}.
  \hyperref[struct_of_contact_set]{Condition $(\mathbf{A})$} excludes
  the possibility of a fractal contact set. Technical details will be
  given in Remark \ref{struct_of_contact_set}, but we note that this
  assumption has been proven to hold in some cases, with the general
  case still an open problem. All original results in this article
  depend upon \hyperref[struct_of_contact_set]{Condition
    $(\mathbf{A})$}, and it alone is sufficient to prove $h$-robust a
  posteriori error control in $\leb{p}(\W)$ for $p \in (4,\infty)$ on
  the positive part of the error (Theorem \ref{thm:l4:pos_bound}).
\item 
  We utilise a dual problem to derive a bound for the negative part of
  the error that depends upon the \emph{discrete} solution. We make an
  assumption on the structure of its contact set, denoted
  \hyperref[def:condition_A_h]{Condition $(\mathbf a_h)$}, which, to
  our knowledge, cannot be proven a priori, but can be verified a
  posteriori once the discrete solution is obtained. Using this, we
  have a posteriori error control on the negative part of the error
  for any discrete solution which satisfies
  \hyperref[def:condition_A_h]{Condition $(\mathbf a_h)$} (Theorem
  \ref{thm:l4:neg_bound}). The bound has an $h$-dependent regularity
  constant which we are able to compute a posteriori (see Remark
  \ref{rem:constant_in_regularity_bound}), and in numerical
  experiments remains small.
\item 
  A strictly stronger assumption, known as
  \hyperref[def:condition_A_h_full]{Condition $(\mathbf{A}_h)$} was
  used in \cite{christof1754finite} to obtain a priori error bounds in
  the $\leb{4}(\W)$ norm using duality methods, implies that the
  $h$-dependent constant mentioned above is uniformly bounded as $h
  \to 0$. Under this stronger assumption, the bound for the negative
  part of the error would therefore be $h$-robust.
\end{enumerate}
}


The rest of the paper is set out as follows: In \S \ref{sec:setup} we
introduce notation, formalise the model problem and describe the
finite element approximation. In \S \ref{sec:apost} we state the key
estimates and main results. In \S \ref{sec:bilateral_approx} we
discuss bound-preserving interpolation from above and below. We
introduce a new interpolant which is shown to preserve bilateral
bounds have optimal approximation properties. The performance of our
error estimates are then tested numerically in \S \ref{sec:num}.

\section{Problem setup}
\label{sec:setup}

This section contains a summary of the problem at hand, as well as a
brief review of the current state of the art for regularity of the
variational inequality and its finite element approximation.  We note
that the article \cite{christof1754finite} considers the problem on
the unit square.  In this work, to make use of their results, we will
do the same, but remark that arguments can be extended to convex
polygonal domains.

In what follows, $\W\subseteq \reals^2$ is assumed to be the unit
square. Let $\w$ be a measurable subset of the domain $\W$, and let
$\leb{p}(\w)$ denote the Lebesgue spaces of $p$-th power Lebesgue
integrable functions over $\w$ with corresponding norm
$\Norm{\cdot}_{\leb{p}(\w)}$. We denote $\ltwop{\cdot}{\cdot}_\w$ as
the $\leb{2}$ inner product over $\w$ and use the convention of dropping
the subscript if $\w=\W$. Finally, we let $\sob{k}{p}(\w)$ be the
Sobolev space whose first $k$ weak derivatives are $\leb{p}(\w)$ and
let $\sobh{k}(\w):=\sob{k}{2}(\w)$ with corresponding norms
$\Norm{\cdot}_{\sob{k}{p}(\w)}$ and $\Norm{\cdot}_{\sobh{k}(\w)}$
respectively. For a real-valued function $g$, we define $g_+ =
\max(g,0)$ and $g_- = \min(g,0)$.

We will also henceforth consistently use the notation that $p$ and
  $q$ denote H\"older conjugates satisfying
  \begin{equation}
    \frac 1p + \frac 1q = 1,
  \end{equation}
  with $p \in (4,\infty)$ and $q \in (1, \frac 43)$.

We define the convex set of test and trial functions
\begin{equation}
  \cK = \{v \in H^1(\W) : v \geq 0 \text{ on } \partial \W \} \subset \sobh1(\W)
\end{equation}
that is appropriate to define the weak form of
(\ref{eq:diff_eq})--(\ref{eq:constraints}). It reads:
let $f \in \leb{s}(\W)$, for $s\geq 2$ and find $u
\in \cK$ such that

\begin{equation} \label{eq:weak_form}
  a(u, v-u)
  :=
  \ltwop{\nabla u}{\nabla v - \nabla u}
  +
  \ltwop{u}{v-u}
  \geq
  \ltwop{f}{v - u}
  =: l(v-u)
  \quad
  \forall v \in \cK.
\end{equation}

This leads us to state some well known properties of the problem
(\ref{eq:diff_eq})--(\ref{eq:constraints}) as well as some new
regularity results which require stronger assumptions on the problem data.

\begin{proposition}[Regularity of the primal problem \label{prop:primal_regularity}
{\cite[ Theorem 3.2.3.1]{Grisvard:1985}}]
  Given $f \in \leb{2}(\W)$ there exists a unique solution $u \in
  H^2(\W)$ to problem \eqref{eq:weak_form}. Moreover
  \begin{equation}
    \Norm{u}_{\sobh{2}(\W)} \leq C \Norm{f}_{\leb{2}(\W)}.
  \end{equation}
\end{proposition}

\begin{remark}
  There are variants of the Signorini problem for which this
  regularity result cannot be applied.  For example, if the boundary
  also contains portions where Neumann and\slash or Dirichlet
  conditions are applied, then the boundary conditions can not be
  expressed in the correct form to apply Proposition
  \ref{prop:primal_regularity}.  In this case we are not guaranteed $u
  \in \sobh{2}(\W)$.  Such problems do arise in practical
  applications.  Indeed they are studied in
  \cite{Blum2000,ashby2021adaptive}.
\end{remark}

\begin{definition}[Contact Set]
  Let $u$ be the unique solution of problem \eqref{eq:weak_form}. Since
  we have $u \in H^2(\W)$ from Proposition \ref{prop:primal_regularity},
  $u$ has a continuous representative on $\overline{\Omega}$ which
  we may use to define the contact set

  \[
  \cA
  :=
  \{ x \in \partial \Omega \mid u(\vec x) = 0\},
  \]
  with $\interior \cA$ the relative interior of this
  set.
\end{definition}

\begin{remark}[{\hyperref[struct_of_contact_set]{Condition $(\mathbf{A})$}}]\label{struct_of_contact_set}

  Stronger regularity for the solution of
    \eqref{eq:weak_form} follows from an assumption referred to as
    \hyperref[struct_of_contact_set]{Condition $(\mathbf{A})$} in \cite{christof1754finite}, namely that
    the contact set, $\cA$, consists of finitely many connected
    components and isolated points, and in particular the set of
    critical points (i.e. points contained in the boundary of the
    contact set) does not contain accumulation points.

 This means that the relative interior, $\interior \cA$, is a union of
 open intervals in $\partial \W$.  This is a common assumption in
 works relating to problems with a Signorini boundary, such as \cite{
   belgacem1999extension,belhachmi2003quadratic, ben2000numerical,
   christof1754finite, haslinger1996numerical,hueber2005optimal}, with
 a related assumption on the free boundary in a problem with an
 interior obstacle given in \cite{brezzi1977error,brezzi1978error}. We
 also mention \cite{drouet2015optimal}, where a priori bounds in
 $\sobh{1}(\W)$ are proved without this assumption, however their
 methods are not applicable to the analysis in this work.

 {Optimal error estimates in certain non-energy norms can be
   obtained for the Signorini problem \emph{if} an appropriate dual
   problem can be defined which has sufficient
   regularity. \hyperref[struct_of_contact_set]{Condition
     $(\mathbf{A})$} is a sufficient condition under which such a dual
   problem can be defined. It is not required for existence and
   uniqueness of the continuous problem, but is rather a regularity
   assumption. If \hyperref[struct_of_contact_set]{Condition
     $(\mathbf{A})$} is not satisfied, optimal a priori estimates in
   the energy norm will still hold \cite{drouet2015optimal}, but
   bounds in low order norms may be rendered suboptimal.}

 In order to violate this assumption, the solution to the Signorini
 problem would have to have infinitely many critical points.  The
 boundary of the contact set would therefore contain accumulation
 points, meaning that approaching such a point, the solution along the
 boundary would switch from contact to non-contact infinitely many
 times and over arbitrarily small distances, see Figure
 \ref{fig:condA}.

  We finally remark that this condition was recently shown to hold
    in \cite{apel2020regularity} in polygonal domains in some cases,
    although the verification in the general setting remains a
    challenging open problem.
\end{remark}

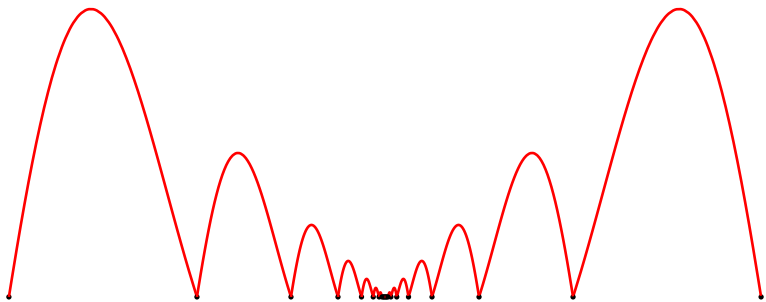
\begin{figure}[h!]
  \centering
  \begin{tikzpicture}[scale=10, line width=0.5pt]
    \tikzmath{
      real \xpos, \xposNext;
      int \n;
      \n = 10;
    }

    \foreach \i in {1,2,...,\n} {
      \tikzmath{
        \xpos = 1/(2^\i);
        \xposNext = 1/(2^(\i+1));
      }
      \fill[black] (-\xpos,0) circle (0.1pt);
      \fill[black] (\xpos,0) circle (0.1pt);
      \draw[domain=-\xpos:-\xposNext, smooth, red, line width=1pt] plot (\x, {abs(\x)*abs(sin(360*(\x+\xpos)/(2*\xposNext)))});
      \draw[domain=\xposNext:\xpos, smooth, red, line width=1pt] plot (\x, {abs(\x)*abs(sin(360*(\x-\xpos)/(2*\xposNext)))});
    }

  \end{tikzpicture}
  \caption{\label{fig:condA} An illustration to show how the trace of a function which violates \hyperref[struct_of_contact_set]{Condition $(\mathbf{A})$} might behave.}
\end{figure}

\begin{proposition}[Improved Regularity
    of the Primal Problem {\cite[Theorem 2.3]{christof1754finite}}]
  \label{prop:better_primal_regularity}
  Suppose that $u$ solves \eqref{eq:weak_form}, and
    that \hyperref[struct_of_contact_set]{Condition $(\mathbf{A})$} holds. If $f \in \leb{s}(\W)$ for $ 2 < s
  < 4$, then $u \in \sob 2s (\W)$. If $f\in \leb{s}(\W)$
    for $ 4 < s < \infty$ then $u$ admits a decomposition $u = u_R +
    u_S$ with $u_R \in \sob{2}{s}(\W)$ and $u_S \in
    \sobh{5/2-\epsilon}(\W)$.
\end{proposition}

\subsection*{Finite element discretisation}

We consider $\T{}$ to be a conforming triangulation of $\W$, namely,
$\T{}$ is a finite family of sets such that
\begin{enumerate}
\item $K\in\T{}$ implies $K$ is an open simplex or box,
\item for any $K,J\in\T{}$ we have that $\closure K\cap\closure J$ is a full
  lower-dimensional simplex (i.e., it is either $\emptyset$, a vertex,
  an edge or the whole of $\closure K$ and $\closure J$) of both
  $\closure K$ and $\closure J$ and
\item $\bigcup_{K\in\T{}}\closure K=\closure\W$.
\end{enumerate}
We will make the assumption that the triangulation is shape-regular. Further, we
define $h: \W\to \reals$ to be the {piecewise constant} \emph{meshsize
function} of $\T{}$ given by
\begin{equation}
  h(\vec{x}):=\max_{\closure K\ni \vec{x}}\text{diam} (K).
\end{equation}

We let $\E{}$ be the skeleton (set of common interfaces) of the
triangulation $\T{}$ and say $e\in\E$ if $e$ is on the interior of
$\W$ and $e\in \partial\W$ if $e$ lies on the boundary $\partial \W$
{and set $h_e$ to be the diameter of $e$.}  We let $\mathcal R^1(K)$
denote a space of piecewise linear (respectively bilinear) polynomials
over a triangle (respectively quadrilateral), that is, $\mathcal
R^1(K) \in \{\mathbb P^1(K)$, $\mathbb Q^1(K)\}$, and introduce the
\emph{finite element space}
\begin{equation}
  \fes := \{\phi \in \sobh1(\W) : \phi|_K \in\mathcal R^1(K)\}
\end{equation}
to be the usual space of continuous piecewise polynomial
functions. 
We define the \emph{element patch} of $K \in \T{}$ as 
\begin{equation*}
  \widehat K 
  :=
  \{J \in \T{} | \overline J \cap \overline K \neq \emptyset\}.
\end{equation*}
We also define a discrete analogue of $\cK$ as
\begin{equation}\label{def:k_h}
  \cK_h := \{v \in \fes : v \geq 0 \text{ on } \partial \W \}
  =
  \cK \cap \fes
  \subset \sobh1(\W)
  .
\end{equation}

We are now in a position to state the finite element approximation of
(\ref{eq:diff_eq})--(\ref{eq:constraints}) which is to find $U\in \cK_h$
such that
\begin{equation} \label{eq:discrete_form}
  a(U, \Phi - U)
  \geq
  l(\Phi - U) \quad \forall \Phi
  \in \cK_h,
\end{equation}
Note that we restrict our attention to either piecewise linear finite
element spaces on triangular meshes or piecewise bilinear spaces on
quadrilateral meshes.  For the sake of this work, we do not
  extend to higher degree polynomial approximation and consider only
  lowest order.

 \begin{proposition}[The discrete problem is well posed  {\cite[Thm 2.1]{kinderlehrer2000introduction}}]
   For $f\in\leb{2}(\W)$ and all $h >0$ there exists a
   unique solution $U$ of \eqref{eq:discrete_form}.
 \end{proposition}

 Under slightly stronger assumptions, namely that $f$
   is essentially bounded and a technical assumption on the discrete
   solution $U$, which we detail below, one can prove a priori estimates in $\leb{4}(\W)$
   \cite{christof1754finite}. We note that this result is included for
   completeness, and that we do not make these additional
   assumptions.
   \begin{definition}[Discrete contact Set]
    The discrete contact set is defined analogously to the contact
    set. With $U$ the solution to \eqref{eq:discrete_form}, we define
    \[
    \cA_U
    :=
    \{ x \in \partial \Omega \mid U(\vec x) = 0\}.
    \]
  \end{definition}

   \begin{definition}[{\hyperref[def:condition_A_h_full]{Condition $(\mathbf{A}_h)$}} \cite{christof1754finite}]\label{def:condition_A_h_full}
    We say that \hyperref[def:condition_A_h_full]{Condition $(\mathbf{A}_h)$} holds if there exist points $d_i$ lying on $\partial \W$ and numbers $\delta_i >0$ such that the intersections of the open balls $B_{\delta_i}(d_i)$ and the boundary have non-zero distance from each other, and from the corners of the domain, such that the following holds for all sufficiently small $h > 0$.
    \begin{enumerate}
      \item
      The sets $B_{\delta_i}(d_i) \cap \partial \W$ cover the boundary of the discrete contact set, and such that each $B_{\delta_i}(d_i) \cap \partial \W$ contains precisely one element of the boundary of the discrete contact set.
      \item
      Every connected component of $\cA_U$ has a non-empty interior.
    \end{enumerate}
    \end{definition}

   We now summarise the range of a priori results and the main ideas used to derive them
   in \cite{christof1754finite} to illustrate the interplay of primal and dual regularity.
   
   \begin{remark}[A priori error bounds {\cite[Propositions 5.5 and 5.7]{christof1754finite}}]
     \label{rem:apriori}
     Let $u$ solve \eqref{eq:weak_form} for some $f \in \leb{p}(\W)$ with
   $p\in (4, \infty)$. Let $R : \sobh 1(\W) \to \fes$ denote the Ritz
   projector and let $U \in \fes$ be the solution of
   \eqref{eq:discrete_form}. Then for all $\varepsilon \in (0,1)$ there
   exists $C>0$ such that the superconvergence result
   \begin{equation}\label{eq:superclose}
     \Norm{R(u) - U}_{\sobh1(\W)} \leq C h^{\tfrac 32 - \tfrac 2p -\varepsilon} 
   \end{equation}
   holds.
   The authors additionally assume that $f \in \leb{\infty}(\W)$.
   With this assumption, using inverse estimates and choosing arbitrarily large $p$, the result \eqref{eq:superclose} provides optimal control of $R(u) - U$ in
   the $\sob{1}{s}(\W)$ norm for $s<4$, i.e.,
   \begin{equation}\label{eq:superclose_2}
     \Norm{R(u) - U}_{\sob{1}{s}(\W)} \leq C h.
   \end{equation}
   Using an error bound in $\sob{1}{\infty}(\W)$, which the authors also prove, a bound in $\sob{1}{4}(\W)$ is obtained from equation \eqref{eq:superclose_2} of optimal order up to arbitrarily small $\varepsilon$, namely
   \begin{equation}
     \Norm{R(u) - U}_{\sob{1}{4}(\W)} \leq C h^{1-\varepsilon}.
   \end{equation}
   This result is combined with classical results on approximability of the Ritz projector to obtain
   \begin{equation}\label{eq:apriori_1_4}
     \Norm{u - U}_{\sob{1}{4}(\W)} \leq C h^{1-\varepsilon}.
   \end{equation}
   The a priori result \eqref{eq:apriori_1_4} is a key ingredient in an Aubin-Nitsche duality
   argument which uses a H\"older inequality to bound in the strongest norm possible for an appropriately defined dual solution, which does not have $\sobh{2}(\W)$ regularity. 
   The bound \eqref{eq:apriori_1_4} is able to compensate for this and deliver optimal rates (up to an arbitrarily small positive number $\varepsilon$):
    \begin{equation}
      \Norm{(u-U)_+}_{\leb{4}(\Omega)}
      \leq
      Ch^{2 - \varepsilon}.
    \end{equation}
    If we additionally assume that \hyperref[def:condition_A_h_full]{Condition $(\mathbf{A}_h)$} holds, then we also have
    \begin{equation}
      \Norm{(u-U)_-}_{\leb{4}(\Omega)}
      \leq
      Ch^{2 - \varepsilon},
    \end{equation}
    and therefore achieve full a priori error control in $\leb{4}(\W)$.  

     Note that convergence rates in $\sob{1}{s}(\W)$ for $s>4$ are suboptimal, and that while the argument above can be modified to include stronger norms at a suboptimal rate, this suboptimality carries through to the $\leb{p}(\W)$ estimate, i.e.,
    \begin{equation}
      \label{eq:subopt}
      \Norm{u-U}_{\leb{p}(\Omega)}
      \leq
      Ch^{\tfrac 32 + \tfrac 2p - \varepsilon}.
    \end{equation}
     
   \end{remark}

\section{A posteriori error control}
\label{sec:apost}

In this section we state and prove the main result of this work, a
rigorous a posteriori bound in $\leb{s}(\W)$ for $s > 4$
for the problem \eqref{eq:diff_eq}.  To begin we
state some technical results.

\begin{proposition}[Trace theorem in $\leb{s}$]\label{lem:trace}
    Let $v\in \sob1s(K)$ with $s\in (1,\infty)$.  Then, there exists a
    constant $C_{\text{tr}}$ depending only on $s$ and $K$ such that
    \begin{equation}
        \Norm{v}_{\leb{s}(\partial K)}
        \leq
        C_{\text{tr}} \qp{h^{-\tfrac 1 s}\Norm{v}_{\leb s(K)}
            +
            h^{1-\tfrac 1 s} \Norm{\nabla v}_{\leb s(K)}
        }.
    \end{equation}
\end{proposition}

\begin{definition}[Jump operator]
    We define jump operators as $\jump{v} = {{{v}|_{K_1} \vec n_{K_1} +
            {v}|_{K_2}} \vec n_{K_2}}$, $\jump{\vec v} = {{\vec{v}|_{K_1}}}
    \cdot \vec n_{K_1} + {{\vec{v}|_{K_2}}} \cdot \vec n_{K_2}$.
\end{definition}

\begin{lemma}[Interpolation estimate]\label{lem:interpolation}
    Let $\cI$ denote the piecewise linear Lagrange interpolator over
    $\T{}$. The following estimate holds for all elements $K \in \T{}$,
    faces $e\in\E$ and all functions $v \in \sob{2}{s}(K)$ with  $s> 1$.
    \begin{align*}
        \max\qp{
            h^{-2}\Norm{v - \cI v}_{\leb{s}(K)}
            ,
            h^{-2+1/s}\Norm{v - \cI v}_{\leb{s}(e)}
        }
        &\leq C \Norm{v}_{\sob{2}{s}(\widehat K)}.
    \end{align*}
\end{lemma}

\subsection{Dual variational inequality}

Motivated by the approach used in \cite{mosco1977error} to derive a priori
$\leb{2}(\W)$ estimates, we define the following set

\begin{equation}
    \cM=
    \{v \in \sobh{1}(\W)\mid v \leq 0 \text{  on  } \interior \cA\}.
\end{equation}
For $p \in (4, \infty)$ consider the problem of finding $z \in \cM$
such that
\begin{equation} \label{eq:mosco_dual}
    a(z, v-z)
    \geq
    \ltwop{\max(0,u-U)^{p-1}}{v-z}
    \quad \forall v \in \cM.
\end{equation}

\begin{remark}\label{rem:l4:admissible_function}
 Observe that $z + u - U$ lies in $\cM$, since by the definition of
 $\interior \cA$, $u$ vanishes there and $U \geq 0$ on $\partial \W$
 by definition of $\cK_h$, \eqref{def:k_h}. 
\end{remark}

To derive a posteriori error bounds, we require results on the
stability and regularity of the dual problem \eqref{eq:mosco_dual}, which we will now quantify.

Before presenting the necessary results, we make the following
definitions regarding the corners of the domain and the boundary
conditions.

\begin{definition}\label{def:l4:corner_functions}
  {We let $\{b_1,...,b_{N_p}\}$ be the set of critical
    points that make up the boundary of $\interior \cA$ and the
    vertices of $\partial \W$.} Note that under the assumption that
    \hyperref[struct_of_contact_set]{Condition $(\mathbf{A})$} holds,
    this is a finite set consisting of the vertices of $\partial \W$
    and the points at which there is a change in boundary condition.
    We let $\omega_k$ denote the angle at the boundary at $b_k$, that
    is, $\omega_k = \pi$ at a change of boundary conditions or
    $\omega_k = \tfrac{\pi}{2}$ at the corners of the domain, since
    the domain is the unit square.  We will denote by $r_k, \theta_k$
    the polar coordinates centred at point $b_k$, $\zeta_k$ a smooth
    cut-off function identically equal to 1 in a neighbourhood of
    $b_k$, and $\phi_k = \phi_k(\theta_k)$ a trigonometric function.
\end{definition}
The functions described in Definition \ref{def:l4:corner_functions}
are the key tools in quantifying the behaviour of the solution around
the points $b_k$. The properties of the dual problem that we will need
are summarised in the following lemma.

\begin{proposition}[Properties of the dual problem
\cite{christof1754finite}]
    \label{pro:reg}
    Let $u$ solve \eqref{eq:weak_form} for some $f \in \leb{2}(\W)$,
    and satisfy \hyperref[struct_of_contact_set]{Condition $(\mathbf{A})$}, and $U \in \fes$ solve
    \eqref{eq:discrete_form}.  Then the dual problem
    \eqref{eq:mosco_dual} is uniquely solvable in $\sobh{1}(\W)$.  In
    addition, we have $z \geq 0$ in $\W$ and $z=0$ on $\cA$, and for
    any $p\in(4,\infty)$ and its conjugate $q = \tfrac{p}{p-1}$
    there exists $C>0$ independent of $h$ such that the stability
    bound
    \begin{equation}\label{eq:dual_stability}
        \Norm{z}_{\sobh{1}(\W)}
        \leq
        C \Norm{\max(0, u-U)^{p-1}}_{\leb{q}(\W)}
    \end{equation}
    holds.
    Furthermore, we have the following regularity result in $\sob 2 q(\W)$
    \begin{equation}\label{eq:dual_regularity}
        \Norm{z}_{\sob{2}{q}(\W)}
        \leq
        C \Norm{\max(0, u-U)^{p-1}}_{\leb{q}(\W)}.
    \end{equation}
\end{proposition}

\begin{proof}

We first claim that $a(z_+, z_-) = 0$.
To see this, note that by definition we have $z=z_+ + z_-$, and
\begin{equation}\label{eq:l4:positive_negative_parts}
  a(z_+, z_-)
  =
  \int_{\W}\grad z_+ \cdot \grad z_- + z_+ z_- \diff x.
\end{equation}
It is clear that $z_+ z_- \equiv 0$ since $z$ cannot be both positive and negative.
Further, we can write $\W = \{x \in \W : z <0 \} \cup \{x \in \W : z \geq 0 \}$.
Note that in each set, precisely at least one of $z_+$ and $z_-$ is identically the zero function, and therefore if the set has nonzero measure, its weak gradient must be zero there.
Contributions to the integral \eqref{eq:l4:positive_negative_parts} must therefore be zero and the claim is proved.

Since $z \in \cM$, $z \leq 0$ on $\interior \cA$, its positive part must be zero there, and we may take $v = z_+$ as test function in \eqref{eq:mosco_dual}, which yields

\begin{equation}
    -\|z_-\|^2_{\sobh{1}(\W)}
    =
    a(z, -z_-)
    \geq
    \ltwop{\max(0,u-U)^{p-1}}{-z_-}.
\end{equation}
In other words, since $\max(0,u-U)^{p-1}$ is non-negative and $z_-$ is non-positive,

\begin{equation}
    0
    \leq
    \|z_-\|^2_{\sobh{1}(\W)}
    \leq
    \ltwop{\max(0,u-U)^{p-1}}{z_-}
    \leq
    0,
\end{equation}
so that $z_- \equiv 0$, which proves that $z\geq 0 $.
This in turn gives $z = 0$ on $\interior \cA$.

The bound \eqref{eq:dual_stability} follows after noting that we can take $v=0$ in \eqref{eq:mosco_dual} to see

\begin{equation}
  a(z, -z)
  \geq
  \ltwop{\max(0,u-U)^{p-1}}{-z},
\end{equation}
and $v=2z$ in \eqref{eq:mosco_dual}, yielding
\begin{equation}
  a(z, z)
  \geq
  \ltwop{\max(0,u-U)^{p-1}}{z}.
\end{equation}
Therefore,
\begin{equation*}
  \begin{split}
  \Norm{z}^2_{\sobh{1}(\W)}
  &=
  \ltwop{\max(0,u-U)^{p-1}}{z} \\
  &\leq
  \Norm{\max(0,u-U)^{p-1}}_{\leb{q}(\W)}
  \Norm{z}_{\leb{p}(\W)} \\
  &\leq
  C_{\text{Sob}}
  \Norm{\max(0,u-U)^{p-1}}_{\leb{q}(\W)}
  \Norm{z}_{\sobh{1}(\W)},
  \end{split}
\end{equation*}
where we have used the embedding of $\sobh{1}(\W)$ in $\leb{p}(\W)$ with embedding constant $C_{\text{Sob}}$.
The result now follows after dividing by $\Norm{z}_{\sobh{1}(\W)}$.

Now, $z$ is the weak solution of the boundary value problem

\begin{equation}
 - \Delta z + z
 =
 \max(0,u-U)^{p-1},
\end{equation}
with zero Dirichlet boundary condition on $\cA$ and zero Neumann condition on the remainder of the boundary.
The problem therefore fits into the framework of \cite[Theorem 4.4.3.7]{Grisvard:1985}, which allows us to quantify the regularity of $z$ to the following extent.
With the notation of Definition \ref{def:l4:corner_functions}, there exist unique real numbers $c_{k,m}$ such that

\begin{equation}\label{eq:l4:Grisvard_expansion}
z - \sum_{\substack{1 \leq k \leq N_p \\ -2 \slash q' < \lambda_{k,m} < 0 \\\lambda_{k,m} \neq -1}} c_{k,m} r^{-\lambda_{k,m}} \zeta_k \phi_k 
\in 
\sob{2}{(4-\varepsilon)\slash (1 - \varepsilon)}(\W),
\end{equation}
where $q'$ is the H\"older conjugate of $(4-\varepsilon)\slash (1 -
\varepsilon)$, $\lambda_{k,m}$ are eigenvalues of a Laplacian operator
which depends upon $\omega_k$ (see \cite[\S 4.4]{Grisvard:1985}
for complete enumeration of the $\lambda_{k, m}$ in all cases).  In
our case where $\W$ is a convex polygonal domain, the regularity of
$z$ is determined by the term in \eqref{eq:l4:Grisvard_expansion} with
the lowest power of $r$, which is $r^{1 \slash 2}
  \zeta_k \phi_k$.  The singularity of type $r^{1 \slash 2}$ occurs
at points where the boundary condition changes (compare with
\cite[\S3]{apel2020regularity} on this point).

Now, we observe that $r^{1 \slash 2} \zeta_k \phi_k \in \sob{k}{s}(\W)$
if and only if $k - \tfrac 2 s < \tfrac{1}{2}$.  This gives a limit on
the regularity of the second derivatives of $z$: $z\in \sob{2}{s}(\W)$
only if $s < \tfrac 4 3$.

The proof of the stability bound in $\sob{2}{q}(\W)$ is a technical
argument that we will not reproduce here, but can be found in full in
\cite[Lemma 5.2]{christof1754finite}.
\end{proof}

\begin{remark}[Smoothness of the dual problem]
  It is important to note that the limit on regularity for the dual
  problem is imposed by the assumption of convexity of the domain and
  the nature of the boundary conditions, not the regularity of the
  problem data $f$.  Indeed, the nonsingular part of $z$, that is, the
  expression in \eqref{eq:l4:Grisvard_expansion} possesses higher
  regularity depending on the problem data, but the regularity of the
  singular components of $u$ is limited by the geometry.  For certain
  boundary conditions including the Signorini condition
  \eqref{eq:constraints}, $\sobh{2}(\W)$-regularity follows from
  Proposition \ref{prop:primal_regularity} which implies that the
  singular part of the solution vanishes.  For the mixed boundary
  condition, this result does not apply, and the regularity is
  therefore constrained by the least regular of the components in
  \eqref{eq:l4:Grisvard_expansion}.
\end{remark}

\begin{remark}[Choice of norm for a posteriori error estimation]
Proposition \ref{pro:reg} also motivates a posteriori error estimation
in $\leb{p}(\W)$ for $p > 4$.  We see that we can only apply a duality
approach for $q < \tfrac 4 3$, with the H\"older conjugate of $\tfrac
4 3$ being $4$.  Thus, $\leb{4 + \varepsilon}(\W)$ for any
$\varepsilon \ll 1$ is the weakest space in which
estimates of the traditional residual form can be established.
In particular, error estimates in $\leb{2}(\W)$ are not possible in this
framework.
\end{remark}

\begin{proposition}
    \label{prop:u_solves_bvp}
    Let $u$ be the solution of (\ref{eq:diff_eq})-(\ref{eq:constraints}) and suppose that $w \in \sobh{1}(\W)$ has tr$(w) = 0$ on $\mathring{\cA}$.
    Then
    \begin{equation}\label{eq:l4:u_solves_bvp}
        a(u, w)
        =
        \ltwop{f}{w}.
    \end{equation}
\end{proposition}

\begin{proof}

    In light of the structure of the contact set, described in Remark
    \ref{struct_of_contact_set}, it can be seen (see \cite[thm 2.3]{christof1754finite}) that $u$ solves the boundary value problem

    \begin{equation} \label{eq:u_bvp}
        \begin{split}
            - \Delta u + u &= f \quad \text{in} \quad \W ,\\
            u &= 0 \quad \text{on} \quad \Gamma_i, i = 1,...,N ,\\
            \grad u \cdot \bm{n} &= 0 \quad \text{on} \quad \Gamma_i, i =
            N+1,...,N+M.
        \end{split}
    \end{equation}
    Here, $\Gamma_i \subset \W$ are disjoint open line segments such that
    $\bigcup_{i=1}^{N+M} \bar{\Gamma}_i = \partial \W$, and such that there is a set
    $R$ of measure zero such that
    \[
    \cA = \bigcup_{i=1}^N \bar{\Gamma}_i  \cup R.
    \]
    The claim now follows immediately by testing \eqref{eq:u_bvp} with $w$
    and integrating by parts.
\end{proof}

\renewcommand{\zbarh}{\binterp(z)}

\subsection{Interpolation operator for bilateral approximation}\label{sec:interpolation_assumption}

We now discuss our requirements for a finite element interpolation
operator, the construction of which is postponed to \S
\ref{sec:bilateral_approx}.

\begin{assumption}[Bilateral optimal approximation]
  \label{ass:bilateral}
  {We make the assumption that there exists an
  interpolation operator $\binterp : \sob{2}{q}(\W) \to \fes$
  such that for any $w\in\sob{2}{q}(\W)$, $q \in (1,
    \tfrac 43)$, with $w \geq 0$,
  \begin{equation}\label{eq:both_sides}
    0 \leq \binterp(w) \leq w
  \end{equation}
  and for any $K\in\T{}$
  \begin{equation}
    \Norm{w - \binterp(w)}_{\leb{q}(K)}
    \leq
    C_b h^2 \norm{w}_{\sob{2}{q}(\patch{K})}.
  \end{equation}}
\end{assumption}
This interpolant is a crucial part of our a posteriori error analysis,
as demonstrated by the following lemma, which allows us to introduce
an interpolant into the error inequality for $(u-U)_+$.  This is done
by using the bilateral bounds \eqref{eq:both_sides} to determine the
sign of $a\qp{\zbarh, u- U}$.

\begin{lemma}
    \label{lem:ineq}
    Let $u$ be the solution of (\ref{eq:diff_eq})-(\ref{eq:constraints}), $U$ be the finite element approximation satisfying (\ref{eq:discrete_form}), $z$ be the dual solution of (\ref{eq:mosco_dual}) and $\zbarh$ be the bilateral approximation of $z$.
    Then

    \begin{equation}\label{interpolant_inequality}
        a\qp{\zbarh, u- U} \leq 0.
    \end{equation}
\end{lemma}

\begin{proof}
    By equation \eqref{eq:l4:u_solves_bvp} and since $0 \leq \zbarh \leq z = 0$ on
    $\mathring{\cA}$,
    we have immediately that
    \begin{equation} \label{eq:sign_1}
        a\qp{u, \, \zbarh}
        =
        \ltwop{f}{ \zbarh}.
    \end{equation}
    Since $\zbarh \geq 0$, we may take $\Phi = U +
    \zbarh$ in \eqref{eq:discrete_form}, giving

    \begin{equation}\label{eq:sign_2}
        a\qp{U, \,- \zbarh}
        \leq
        -\ltwop{f}{ \zbarh} .
    \end{equation}
    The result then follows after combining \eqref{eq:sign_1} and
    \eqref{eq:sign_2}.
\end{proof}

\subsection{A posteriori error bounds}

We are now ready to state and prove the main result of this article, a rigorous a posteriori bound in $\leb{p}(\W)$ for the problem (\ref{eq:diff_eq}) - (\ref{eq:constraints}).
We prove this in two parts, bounding separately the quantities $\Norm{(u-U)_+}_{\leb{p}(\W)}$ and $\Norm{(u-U)_-}_{\leb{p}(\W)}$.

\begin{Theorem}[A posteriori upper bound for the positive part of the error]\label{thm:l4:pos_bound}
    Let $u$ solve problem \eqref{eq:weak_form} and $U$ be the finite
    element approximation.  Then, for $p\in\left(4,\infty\right)$, $q
    = \tfrac{p}{p-1}$ its conjugate and $f\in\leb{p}(\W)$, we have
    \begin{equation}\label{eq:l4:pos_bound}
        \Norm{(u-U)_+}_{\leb{p}(\W)}^p
        \leq
        \eta (U, f, h)
        :=
        C
        \sum_{K\in \T{}} \left(
        \eta_K^p
        +
        \frac 1 2 \eta_J^p
        \right),
    \end{equation}
    where
    \begin{equation}
        \begin{split}
            \eta_K &= h^2 \Norm{-\Delta U + U - f}_{\leb{p}(K)},
            \\
            \eta_J &= h^{2-1/q} \Norm{\jump{\grad U}}_{\leb{p}(\partial K)}.
        \end{split}
    \end{equation}
\end{Theorem}

\begin{proof}

    We may select $v = z + u - U$ in \eqref{eq:mosco_dual}; this function is shown to belong to $\cM$ in Remark \ref{rem:l4:admissible_function}.
    We therefore have
    \begin{equation} \label{eq:err_rep}
        \Norm{(u - U)_+}^p_{\leb{p}(\W)}
        \leq
        a(z, u - U).
    \end{equation}
    We may use Lemma \ref{lem:ineq} to introduce $\zbarh$ as follows:

    \begin{equation} \label{eq:l4_error_representation}
        \Norm{(u - U)_+}^p_{\leb{p}(\W)}
        \leq
        a\qp{z-\zbarh, u - U} .
    \end{equation}
    Since $z$ and $\zbarh$ both have zero trace on $\cA$, we can use Proposition
    \ref{prop:u_solves_bvp} to arrive at

    \begin{equation*}
        \Norm{(u - U)_+}^p_{\leb{p}(\W)}
        \leq
        l\qp{z - \zbarh}
        -
        a\qp{z-\zbarh, U} .
    \end{equation*}
    The rest of the argument follows standard a posteriori techniques, to begin
    \begin{equation}
        \label{eq:up1}
        \begin{split}
            l\qp{z-\zbarh} - a\qp{U,z - \zbarh}
            &=
            \int_\W f\qp{z - \zbarh} - \grad U \cdot \qp{\grad z - \grad \zbarh}
            - U \qp{z - \zbarh} \diff x
            \\
            &=
            \int_\W \qp{f + \Delta U - U} \qp{z - \zbarh} \diff x
            -
            \int_{\E\cup \partial \W} \jump{\grad U} \qp{z - \zbarh} \diff S
            \\
            &=: R_1 + R_2.
        \end{split}
    \end{equation}
    Splitting the integral elementwise and making use of the Cauchy-Schwarz inequality, we see that

    \begin{equation}
        \label{eq:pf2}
        \begin{split}
            R_1
            &=
            \int_\W \qp{f + \Delta U - U} \qp{z - \zbarh} \diff x
            \\
            &\leq
            \sum_{K\in\T{}}
            \Norm{h^2\qp{f + \Delta U - U}}_{\leb{p}(K)}
            \Norm{h^{-2}\qp{z - \zbarh}}_{\leb{q}(K)}.
        \end{split}
    \end{equation}
    Invoking the optimal approximation result of Assumption
    \ref{ass:bilateral} (proven in Lemma
    \ref{prop:one_sided_optimal_approx}), we obtain
    \begin{equation}\label{eq:up2}
      R_1
      \leq
      C_b \sum_{K\in\T{}}
      \Norm{h^2\qp{f + \Delta U - U}}_{\leb{p}(K)}
      |z|_{\sob{2}{q}(\widehat K)}.
    \end{equation}
    Similarly, we bound $R_2$,
    \begin{equation}
        \label{eq:pf3}
        \begin{split}
            R_2
            &=
            -\sum_{e\in\E \cup \partial \W} \int_e \jump{\grad U} \qp{z - \zbarh}
            \\
            &\leq
            \frac{1}{2}
            \sum_{K\in\T{}} \qp{
                \sum_{e\in \partial K}
                \Norm{h^{2-1/q}\jump{\grad U}}_{\leb{p}(e)}
                \Norm{h^{-2+1/q} \qp{z - \zbarh}}_{\leb{q}(e)}
            }.
        \end{split}
    \end{equation}
    To bound this term we use split $z - \zbarh = (z - \mathcal{I}z) +
    (\mathcal{I}z - \zbarh)$ and use the triangle inequality.  We may
    then use approximation properties of the Lagrange interpolant on
    element boundaries, trace estimates (see Proposition
    \ref{lem:trace}), inverse estimates and Assumption
    \ref{ass:bilateral} to finally arrive at
    \begin{equation}\label{eq:pf4}
      R_2
      \leq
      C_b C_{\text{tr}}(1 + C_{\text{inv}})
      \frac{1}{2}
      \sum_{K\in\T{}}
          \Norm{h^{2-1/q}\jump{\grad U}}_{\leb{p}(\partial K)}
          \norm{z}_{\sob{2}{q}(\widehat{K})},
    \end{equation}
    where $C_{\text{inv}}$ is the constant from the inverse inequality.
    Collecting (\ref{eq:up2})--(\ref{eq:pf4}), we have 
    \begin{equation}
        \Norm{(u - U)_+}^p_{\leb{p}(\W)}
        \leq 
        l\qp{z-\zbarh} - a\qp{U,z - \zbarh}
        \leq
        C \sum_{K\in\T{}} \qp{\eta_{K} + \frac{1}{2}\eta_{J}}|z|_{\sob{2}{q}(\widehat{K})}.
    \end{equation}

    Hence, the result follows from a discrete Cauchy-Schwarz inequality and the regularity bound on $z$ given in equation \eqref{eq:dual_regularity}, after noting that, since $1 \slash p + 1\slash q = 1$, 
    \begin{equation*}
        \Norm{(u-U)^{p-1}_+}_{\leb{q}(\W)} = \Norm{(u-U)_+}^{p-1}_{\leb{p}(\W)}.
    \end{equation*}
\end{proof}

We now prove that the negative part of the error also satisfies a bound of the form \eqref{eq:l4:pos_bound}. Once again taking an analogous approach to that in
\cite{mosco1977error}, we define a second dual problem as follows:
let $\bar{\cM} = \{v \mid v \geq 0 \text{ on } \cA_{U} \}$.  Find
$\bar{z} \in \bar{\cM}$ such that

\begin{equation} \label{eq:negative_dual}
    a (\bar{z}, v - \bar{z})
    \geq
    \ltwop{-\max(0, U - u)^{p-1}}{v-\bar{z}}
    \quad
    \forall v \in \bar{\cM}.
\end{equation}

Key properties follow in a manner analogous to Proposition
\ref{pro:reg}, and are summarised in Proposition
\ref{pro:negative_reg}.  This result relies on a mild assumption, a
posteriori verifiable from the discrete solution $U$, which we detail
below.

\begin{definition}[{\hyperref[def:condition_A_h]{Condition $(\mathbf a_h)$}}]\label{def:condition_A_h}
{For the finite element approximation $U$, we say \hyperref[def:condition_A_h]{Condition $(\mathbf a_h)$} holds if} there exist {finitely many} points $d_i$ lying on $\partial \W$ and numbers $\delta_i >0$ such that the intersections of the open balls $B_{\delta_i}(d_i)$ and the boundary have non-zero distance from each other, and from the corners of the domain, and such that
\begin{enumerate}
  \item
  The sets $B_{\delta_i}(d_i) \cap \partial \W$ cover the boundary of the discrete contact set, and such that each $B_{\delta_i}(d_i) \cap \partial \W$ contains precisely one element of the boundary of the discrete contact set.
  \item
  Every connected component of $\cA_U$ has a non-empty interior.
\end{enumerate}
\end{definition}
{We now have three related conditions: \hyperref[struct_of_contact_set]{Condition $(\mathbf{A})$}, \hyperref[def:condition_A_h_full]{Condition $(\mathbf{A}_h)$} and \hyperref[def:condition_A_h]{Condition $(\mathbf a_h)$} The latter two conditions are not discrete analogues of the first; indeed any finite element function may only have a contact set with finitely many connected components. Rather they place some restrictions on the boundaries of discrete contact sets. We note that in the continuous case, the contact set is allowed to contain singletons, whereas this is not allowed for discrete solutions. In addition, \hyperref[def:condition_A_h_full]{Condition $(\mathbf{A}_h)$} assumes topological stability of the discrete contact set as $h \to 0$.}

\begin{remark}[Verification of Condition $(\mathbf{a}_h)$]
  \label{rem:verification_of_condition}
  Notice that the \hyperref[def:condition_A_h]{Condition
      $(\mathbf{a}_h)$} is strictly weaker than its a priori sibling
     \hyperref[def:condition_A_h_full]{Condition
      $(\mathbf{A}_h)$} where in addition, the topology of the
    discrete contact set is assumed to be stable in the limit $h \to
    0$.  \hyperref[def:condition_A_h]{Condition $(\mathbf{a}_h)$}
    ensures that the discrete contact set does not include isolated
    singletons, and that the boundary of the discrete contact set does
    not include any of the corners of the domain. This condition can
    be verified a posteriori from the discrete solution $U$ by
    examination of the boundary nodal values. {For example, since $U$ is piecewise linear on the boundary, a boundary degree of freedom is a singleton component of $\cA_U$ if and only if it is zero and its neighbouring boundary degrees of freedom are both non-zero. Likewise, a degree of freedom at a corner of the domain is in the boundary of $\cA_U$ if it is zero and precisely one of its neighbours is non-zero. Thus, one can iterate over boundary degrees of freedom and if neither of these situations occur, then \hyperref[def:condition_A_h]{Condition
    $(\mathbf{a}_h)$} holds for $U$.}
    
    We remark that this
    condition is only required to prove bounds on the negative part of
    the error. This is due to the different convex sets in which
    solutions are sought for the two dual problems. The set $\cM$ is
    defined using the contact set of $u$, and is therefore
    $h$-independent, whereas the definition of $\bar{\cM}$ involves
    the discrete contact set.  
\end{remark}

\begin{proposition}[Properties of dual problem for the negative part of the error \cite{christof1754finite}]
  \label{pro:negative_reg}
  Let $u$ solve \eqref{eq:weak_form} for some $f \in \leb{2}(\W)$, and
  $U \in \fes$ solve \eqref{eq:discrete_form}, with
    $U$ satisfying \hyperref[def:condition_A_h]{Condition
      $(\mathbf{a}_h)$}.  Then the dual problem
  \eqref{eq:negative_dual} is uniquely solvable in $\sobh{1}(\W)$.  In
  addition we have $\bar{z} \leq 0$ in $\W$ and $\bar{z}=0$ on
  $\cA_{U}$.  For $p\in (4,\infty)$ and its conjugate $q =
  \tfrac{p}{p-1}$ there exists $C>0$ such that the
  stability bound
  \begin{equation}\label{eq:negative_dual_stability}
      \Norm{\bar{z}}_{\sobh{1}(\W)}
      \leq
      C \Norm{\max(0, U-u)^{p-1}}_{\leb{q}(\W)}
  \end{equation}
  holds.  In addition, we have the following
  regularity result 
  \begin{equation}\label{eq:negative_dual_regularity}
      \Norm{\bar{z}}_{\sob{2}{q}(\W)}
      \leq
      C \Norm{\max(0, U-u)^{p-1}}_{\leb{q}(\W)}.
  \end{equation}
\end{proposition}

\begin{remark}[The constant in Proposition \ref{pro:negative_reg}]
  \label{rem:constant_in_regularity_bound}
  { A closer examination of the proof of Proposition \ref{pro:negative_reg}, given in \cite[Lemma 5.6]{christof1754finite} reveals that the constant in Equation \eqref{eq:negative_dual_regularity} hides a dependence on the number of boundary points of the contact set of the discrete solution. To better quantify this regularity constant, we give some details of the proof. For some $h$, under the given assumptions there are finitely many, $N_h$, say, critical points
  (i.e. the boundary points of $\cA_U$). If $\psi_i$ is a smooth cut-off function that is identically 1 in a neighbourhood of the $i$th critical point for $i = 1,...,d$ and $\psi_0 = 1-\sum_{i=1}^{N_h} \psi_i$, such that their supports are disjoint, then $\bar z$ admits a decomposition}
  
  \begin{equation*}
    {
    \bar{z} = z_0 + z_1 + ... + z_{N_h},}
  \end{equation*}
  {where for each $i$ we have defined $z_i = \psi_i \bar{z}$.
  Each locally supported function is shown to
  satisfy a regularity bound}
  \begin{equation*}
    {\Norm{z_i}_{\sob{2}{q}(\W)} 
    \leq
    C_i \Norm{\max(0,U - u)^{p-1}}_{\leb{q}(\W)}}
  \end{equation*}
  {and therefore }
  \begin{equation*}
  {  \Norm{\bar{z}}_{\sob{2}{q}(\W)}  
    \leq 
    \sum_i C_i \Norm{\max(0,U - u)^{p-1}}_{\leb{q}(\W)}
    \leq (N_h+1) C \Norm{\max(0,U - u)^{p-1}}_{\leb{q}(\W)},}
  \end{equation*}
  {where $C$ is the maximum of the $C_i$.The quality of the upper bound therefore depends upon the behaviour of $N_h$ as $h \to 0$. Typically it is observed in numerical experiments that the number and location of critical points converge to their true values (see for example \cite{christof1754finite, steinbach_trace} and \S\ref{sec:num}). In \cite[\S 7]{steinbach_trace}, the distance between a discrete critical point and its true location was always bounded from above by the mesh size.  }
  
  {If $N_h$ changes significantly under mesh refinement, as noted in Remark
  \ref{rem:verification_of_condition} once the finite element approximation is
  known we are able to determine the number of critical points by examining the
  boundary degrees of freedom (see Remark \ref{rem:verification_of_condition}) so that if $N_h$ does increase, its expected increase in size can be quantified. If $N_h$ were to increase without limit, this would imply that discrete solutions can oscillate between contact and non-contact over arbitrarily small distances, behaviour that is not present at the continuous level in light of the assumed \hyperref[struct_of_contact_set]{Condition $(\mathbf{A})$}.The stronger \hyperref[def:condition_A_h_full]{Condition $(\mathbf{A}_h)$} assumes that this number remains uniformly bounded as $h \to 0$. Under this assumption there exists a single constant such that the results of Proposition \ref{pro:negative_reg} are true for all sufficiently small $h$.}
\end{remark}

Since $\bar z \leq 0$, we may now once again use the two-sided
approximation.  Let $\widecheck{\Pi}(\bar{z}) = -
  \binterp(-\bar z)$. Then by Assumption \ref{ass:bilateral} (proven
  in \S \ref{sec:bilateral_approx}), $\widecheck{\Pi}(\bar{z})$
  is a finite element function such that $\bar{z} \leq
  \widecheck{\Pi}(\bar{z}) \leq 0$ with required local approximation
  properties.

\begin{Theorem}[A posteriori upper bound for the negative part of the error]\label{thm:l4:neg_bound}
     Let $u$ solve problem \eqref{eq:weak_form} and $U$ be the FE
     approximation, with $U$ satisfying
       \hyperref[def:condition_A_h]{Condition
         $(\mathbf{a}_h)$}. Then, for $p\in\left(4,\infty\right)$, $q
     = \tfrac{p}{p-1}$ its conjugate and $f\in\leb{p}(\W)$ we have
     \begin{equation}\label{eq:neg_bound}
         \Norm{(u-U)_-}_{\leb{p}(\W)}^p
         \leq
         \eta (U, f, h)
         :=
         C
         \sum_{K\in \T{}}
         \left(
         \eta_K^p
         +
         \frac 1 2 \eta_J^p
         \right),
     \end{equation}
     where
     \begin{equation}
         \begin{split}
             \eta_K &= h^2 \Norm{-\Delta U + U - f}_{\leb{p}(K)},
             \\
             \eta_J &= h^{2-1/q} \Norm{\jump{\grad U}}_{\leb{p}(\partial K)}.
         \end{split}
     \end{equation}
\end{Theorem}

\begin{proof}
We may immediately take
$v = \bar{z} + u - U$ in the dual problem since $U=0$ and $u \geq 0$
on $\cA_{U}$.  We therefore have

\begin{equation} \label{eq:negative_error_inequality}
  \int_{\W} \max(0, U-u)^p
  =
  \Norm{(u-U)_-}^p_{\leb{p}(\W)}
  \leq
  a(u - U, \bar{z} ).
\end{equation}

Since we know that $U=0$ on $\cA_{U}$ and
non-negative on $\partial \W$, and since we also have $\bar{z} = 0$ on
$\cA_{U}$ and non-positive on $\Omega$, we can choose
sufficiently small $s>0$ so that $U \pm s \widecheck{\Pi}(\bar{z}) \in \cK_h$.
Taking $\Phi = U \pm s \widecheck{\Pi}(\bar{z})$ as test functions in the discrete problem \eqref{eq:discrete_form} gives

\begin{equation}
    \begin{split}
        a\qp{U, \widecheck{\Pi}(\bar{z})}
        &\geq
        \ltwop{f}{\widecheck{\Pi}(\bar{z})}, \\
        a\qp{U, -\widecheck{\Pi}(\bar{z})}
        &\geq
        -\ltwop{f}{\widecheck{\Pi}(\bar{z})},
    \end{split}
\end{equation}
and therefore we must have

\begin{equation}
    a\qp{U, \widecheck{\Pi}(\bar{z})}
    =
    \ltwop{f}{\widecheck{\Pi}(\bar{z})}.
\end{equation}
Choosing $v = u - \widecheck{\Pi}(\bar{z}) \in \cK$ in \eqref{eq:weak_form} we also
have

\begin{equation}
     a\qp{u, -\widecheck{\Pi}(\bar{z})}
    \geq
    -\ltwop{f}{\widecheck{\Pi}(\bar{z})}.
\end{equation}
Together with \eqref{eq:negative_error_inequality}, we have shown

\begin{equation}\label{eq:negative_error_inequality_2}
    \Norm{(u-U)_-}^p_{\leb{p}(\W)}
  \leq
  a\qp{u - U, \bar{z} - \widecheck{\Pi}(\bar{z})}.
\end{equation}

To introduce the problem data, $f$, of the primal problem, we choose $v=u + \widecheck{\Pi}(\bar{z}) - \bar{z}$ as test function in \eqref{eq:weak_form}.
Since $\widecheck{\Pi}(\bar{z})$ lies between the (non-positive) $\bar{z}$ and zero, this function lies in $\cK$, and we obtain

\begin{equation}
  a\qp{u,\widecheck{\Pi}(\bar{z}) - \bar{z}}
  \geq
  \ltwop{f}{\widecheck{\Pi}(\bar{z}) - \bar{z}},
\end{equation}
or, equivalently.

\begin{equation}
  a\qp{u,  \bar{z} - \widecheck{\Pi}(\bar{z})}
  \leq
  \ltwop{f}{\bar{z} - \widecheck{\Pi}(\bar{z})}.
\end{equation}
Inserting this in \eqref{eq:negative_error_inequality_2}, we now have

\begin{equation}
  \Norm{(u-U)_-}^p_{\leb{p}(\W)}
  \leq
  \ltwop{f}{\bar{z} - \widecheck{\Pi}(\bar{z})}
  -
  a\qp{U, \bar{z} - \widecheck{\Pi}(\bar{z})}.
\end{equation}
The proof now proceeds exactly as in Theorem \ref{thm:l4:pos_bound}.
\end{proof}

\section{Two-sided approximation} \label{sec:bilateral_approx}

In this section we develop the bound preserving interpolant used to prove the a posteriori error bounds in \S \ref{sec:apost}. We construct an interpolant that is bounded from above by a particular function in \S \ref{sec:from_above}. Then, in \S \ref{sec:bilateral}, we introduce a new interpolant that satisfies bilateral bounds and prove that it enjoys the same optimal approximation properties as the Lagrange interpolant.

Let $\{ x_i \}_{i=1}^N$ denote the vertices of the triangulation
$\T{}$ and let $\phi_i \in \fes$ be the $i$-th canonical basis function,
with $\phi_i(x_j) = \delta_{ij}$ for $i,j = 1, \dots, N$. Let
\begin{equation}
  \widehat x_j := \bigcup \{ K\in\T{} : \text{supp}(\phi_j) \cap K \neq \emptyset \}.
\end{equation}
We also recall that $\cI$ denotes the (piecewise linear) nodal Lagrange interpolation operator.
Let $q \in (1, \tfrac 43)$ and assume we have a function $0 \leq
z\in\sob{2}{q}(\W)$.  Then the aim of this section is to examine the
question of existence of an interpolation operator $\binterp :
\sob{2}{q}(\W) \to \fes$ satisfying both
\begin{equation}
  \begin{split}
    0 \leq \binterp(z) &\leq z \text{ and }
    \\
    \Norm{z - \binterp(z)}_{\leb{q}(\W)}
    &\leq
    C h^2 \Norm{z}_{\sob{2}{q}(\W)}.
  \end{split}
\end{equation}

In other words, we wish to show that there exists a two-sided
approximation that also has the optimal approximation properties
enjoyed by the Lagrange interpolant (see Proposition
\ref{lem:interpolation}).  This form of approximation theory arises in
many areas of finite element analysis. In particular, the optimal
approximation of non-smooth functions is a non trivial task relying on
local averaging, see Cl\'ement \cite{Clement1975}, later extended to
zero boundary values in \cite{scottzhang1990finite}.

\subsection{From below}

In this section we discuss positivity preserving interpolation. We note that since we work in two spatial dimensions, and since $q\in
  (1,\tfrac 43)$, point
  values are well defined in $\sob{2}{q}(\W)$.
In this case, one can simply consider the Lagrange interpolant, which
has the properties
\begin{equation}
  \cI z \geq 0 \text{ if } z \geq 0,
\end{equation}
see Figure \ref{fig:lag}, as well as optimal approximation locally and
globally:
\begin{equation}
  \Norm{z - \cI z}_{\leb{q}(\W)} \leq C h^2 \Norm{z}_{\sob{2}{q}(\W)}.
\end{equation}

\begin{figure}[h!]
  \begin{subfigure}{0.3\textwidth}
    \begin{center}
      \begin{tikzpicture}[thick,scale=0.5, declare function={f(\x)=0.3*(\x-3.5)^3-0.237*\x^2+7;a=1;b=6;c=4.94;}]
        \draw[-stealth] (-0.5,0) -- (6.5,0);
        \draw[-stealth] (0,-0.5) -- (0,6.5);
        \pgfmathsetmacro{\lef}{(3*a+b)/4};
        \pgfmathsetmacro{\mid}{(a+b)/2};
        \pgfmathsetmacro{\rig}{(3*b+a)/4};
        \draw[blue] plot[smooth,domain=1:6] ({\x},{f(\x)});
        \draw[dashed] (a,0) node[below]{$a$} |- (0,{f(a)}) node[left] {$0$};
        \draw[dashed] (b,0) node[below]{$b$} -- (b, {f(\lef)});
        \draw[dashed] (b, {f(\lef)}) -- (0, {f(\lef)}) node[left]{${\max(z)}$}  ;
        \draw ({a},{f(a)}) -- ({\lef},{f(\lef)});
        \draw ({\lef},{f(\lef)}) -- ({\mid},{f(\mid)});
        \draw ({\mid},{f(\mid)}) -- ({\rig},{f(\rig)});
        \draw ({\rig},{f(\rig)}) -- ({b},{f(b)});
        \node at ({a},{f(a)}) {\pgfuseplotmark{*}};
        \node at ({\lef},{f(\lef)}) {\pgfuseplotmark{*}};
        \node at ({\mid},{f(\mid)}) {\pgfuseplotmark{*}};
        \node at ({\rig},{f(\rig)}) {\pgfuseplotmark{*}};
        \node at ({b},{f(b)}) {\pgfuseplotmark{*}};
        \draw[dashed,name path=hori] (a,{f(a)}) -- (b,{f(a)});
      \end{tikzpicture}
    \end{center}
    \caption{\label{fig:lag} Approximation with the piecewise linear
      Lagrange interpolant, $\cI(z)$. Notice it is
      positive. \phantom{this text ensures the figure aligns
        properly. EVEN MORE HERE}}
  \end{subfigure}
  \hfill
  \begin{subfigure}{0.3\textwidth}
    \begin{center}
      \begin{tikzpicture}[thick,scale=0.5, declare function={f(\x)=0.3*(\x-3.5)^3-0.237*\x^2+7;a=1;b=6;c=4.94;}]
        \draw[-stealth] (-0.5,0) -- (6.5,0);
        \draw[-stealth] (0,-0.5) -- (0,6.5);
        \pgfmathsetmacro{\lef}{(3*a+b)/4};
        \pgfmathsetmacro{\mid}{(a+b)/2};
        \pgfmathsetmacro{\rig}{(3*b+a)/4};
        \draw[blue] plot[smooth,domain=1:6] ({\x},{f(\x)});
        \draw[dashed] (a,0) node[below]{$a$} |- (0,{f(a)}) node[left] {$0$};
        \draw[dashed] (b,0) node[below]{$b$} -- (b, {f(\lef)});
        \draw[dashed] (b, {f(\lef)}) -- (0, {f(\lef)}) node[left]{$\max(z)$}  ;
        \draw ({a},{f(a)}) -- ({\lef},{f(\lef)});
        \draw ({\lef},{f(\lef)}) -- ({\mid},{f(\mid)-0.1});
        \draw ({\mid},{f(\mid)-0.1}) -- ({\rig},{f(\rig)-0.6});
        \draw ({\rig},{f(\rig)-0.6}) -- ({b},{f(b)-0.6});
        \node at ({a},{f(a)}) {\pgfuseplotmark{*}};
        \node at ({\lef},{f(\lef)}) {\pgfuseplotmark{*}};
        \node at ({\mid},{f(\mid)-0.1}) {\pgfuseplotmark{*}};
        \node at ({\rig},{f(\rig)-0.6}) {\pgfuseplotmark{*}};
        \node at ({b},{f(b)-0.6}) {\pgfuseplotmark{*}};
        \draw[dashed,name path=hori] (a,{f(a)}) -- (b,{f(a)});
      \end{tikzpicture}
    \end{center}
    \caption{\label{fig:lagmod} Approximation with a nonlinear
      interpolant $\interp(z)$. Notice it is bounded above by the
      function, but is no longer bounded below by zero.}
  \end{subfigure}
  \hfill
  \begin{subfigure}{0.3\textwidth}
    \begin{center}
      \begin{tikzpicture}[thick,scale=0.5,declare function={f(\x)=0.3*(\x-3.5)^3-0.237*\x^2+7;a=1;b=6;c=4.94;}]
        \draw[-stealth] (-0.5,0) -- (6.5,0);
        \draw[-stealth] (0,-0.5) -- (0,6.5);
        \pgfmathsetmacro{\lef}{(3*a+b)/4};
        \pgfmathsetmacro{\mid}{(a+b)/2};
        \pgfmathsetmacro{\rig}{(3*b+a)/4};
        \draw[blue] plot[smooth,domain=1:6] ({\x},{f(\x)});
        \draw[dashed] (a,0) node[below]{$a$} |- (0,{f(a)}) node[left] {$0$};
        \draw[dashed] (b,0) node[below]{$b$} -- (b, {f(\lef)});
        \draw[dashed] (b, {f(\lef)}) -- (0, {f(\lef)}) node[left]{$\max(z)$}  ;
        \draw ({a},{f(a)}) -- ({\lef},{f(\lef)});
        \draw ({\lef},{f(\lef)}) -- ({\mid},{f(a)});
        \draw ({\mid},{f(a)}) -- ({\rig},{f(a)});
        \draw ({\rig},{f(a)}) -- ({b},{f(a)});
        \node at ({a},{f(a)}) {\pgfuseplotmark{*}};
        \node at ({\lef},{f(\lef)}) {\pgfuseplotmark{*}};
        \node at ({\mid},{f(a)}) {\pgfuseplotmark{*}};
        \node at ({\rig},{f(a)}) {\pgfuseplotmark{*}};
        \node at ({b},{f(a)}) {\pgfuseplotmark{*}};
        \draw[dashed,name path=hori] (a,{f(a)}) -- (b,{f(a)});
      \end{tikzpicture}
    \end{center}
    \caption{\label{fig:bilat} Approximation with the modified
      bilateral approximation $\binterp(z)$. Notice that $0 \leq
      \binterp(z) \leq z$. \phantom{this text ensures the figure aligns
        properly.}}
  \end{subfigure}
  \caption{An illustration of three piecewise linear operators, black,
    applied to the same function, blue, that satisfies $z>0$. The
    Lagrange interpolant, $\cI(z) \geq 0$, the nonlinear interpolant
    $\interp(z) \leq z$ and the bilateral approximation $0 \leq
    \binterp(z) \leq z$. }
\end{figure}

If point values do not exist, the approximation is
considerably more involved, although has been tackled in
\cite{ChenObstacle2000} where an interpolant is constructed through
local mean-values of the function. It was shown in
\cite{ChenObstacle2000} that the constructed interpolant is
$\leb{q}$-stable, second order accurate and linear. See also
\cite{Veeser:2016,Veeser:2019} for related ideas using a nonlinear
interpolation operator \cite{DeVore:1998}.

\subsection{From above}\label{sec:from_above}

In this section we examine the design of interpolants that are bounded
from above by the function to be interpolated. This means that we seek an
interpolation operator $\interp$ that satisfies
$\interp(z) \leq z$. In this case we can see the Lagrange interpolant
does not satisfy the requirements, see Figure \ref{fig:lag}. Indeed,
any strictly convex function will see the Lagrange interpolant violate
this bound. It can, however, be modified. Indeed, consider the
function:
\begin{equation}
  \label{eq:nonlin}
  \interp(z)(x) = \sum_{i=1}^N \qp{\cI z(x_i)
    -
    \max_{y\in \widehat x_i} (\cI(z)(y) - z(y))}\phi_i(x) =: \cI(z)(x) - R(x).
\end{equation}

\begin{lemma}
  For $0 \leq z \in \sob{2}{q}(\W)$ the approximation $\interp(z)$
  defined in (\ref{eq:nonlin}) satisfies
  \begin{equation}
    \interp(z) \leq z \text{ in } \W.
  \end{equation}
\end{lemma}
\begin{proof}
  Note that by definition of $R$ we have
  \begin{equation}
    \interp(z) = \cI(z) - R \leq \cI(z) - (\cI(z) - z) = z.
  \end{equation}
\end{proof}
This is a nonlinear interpolant since for general
$u,v\in\sob{2}{q}(\W)$
\begin{equation}
  \interp(u + v) \neq \interp(u) + \interp(v).
\end{equation}
This means we cannot directly apply Bramble-Hilbert to obtain optimal
approximation under minimal regularity. Instead we must take a
different approach.


\begin{Theorem}[Optimal approximation]
  \label{the:optimalapprox}
  Suppose $z\in\sob{2}{q}(\W)$, for $q\in (1,\tfrac 43)$. Then there
  exists $C > 0$ independent of $h$ such that for all elements
  $K\in\T{}$ the approximation $\interp(z)$ defined through
  (\ref{eq:nonlin}) satisfies
  \begin{equation}
    \Norm{z - \interp(z)}_{\leb{q}(K)}
    \leq
    C h^2 \norm{z}_{\sob{2}{q}(\widehat{K})}.
  \end{equation}
\end{Theorem}
\begin{proof}
  To begin, we note
  \begin{equation}\label{eq:split_error}
    \Norm{z - \interp(z)}_{\leb{q}(K)}
    \leq
    \Norm{z - \cI(z)}_{\leb{q}(K)}
    +
    \Norm{R}_{\leb{q}(K)}.
  \end{equation}
  we have from Theorem \ref{lem:interpolation} that
  \begin{equation}\label{eq:interp_part}
    \Norm{z - \cI(z)}_{\leb{q}(K)}
    \leq
    C h^2 \norm{z}_{\sob{2}{q}(K)}.
  \end{equation}
  To control the second term, notice that since $R|_K \in
  \mathcal{R}^1(K)$, for a $K\in\T{}$ we can write a lumped
  $\leb{q}$-norm to see that (cf. \cite[Proposition
    12.5]{ern2021finite})
  \begin{equation}
    \label{eq:Rterm}
    \begin{split}
      \Norm{R}_{\leb{q}(K)}^q
      &\leq
      C h^2 \sum_{x_i \in K} \norm{R(x_i)}^q
      \\
      &\leq
      C{ h^2} \sum_{x_i \in K} \Norm{\cI(z) - z}_{\leb{\infty}(\widehat x_i)}^q.
    \end{split}
  \end{equation}
  
  Now, a further property of the Lagrange interpolant (see  \cite[Corollary 4.4.7]{Brenner2008}) is
  \begin{equation}
    \Norm{z - \cI(z)}_{\leb{\infty}(K)}
    \leq
    C h^{2-\tfrac 2 q} \norm{z}_{\sob{2}{q}(K)}.
  \end{equation}
  Hence substituting into (\ref{eq:Rterm}) we have
  \begin{equation}
    \begin{split}
      \Norm{R}_{\leb{q}(K)}^q
      \leq
      C h^{2q} \norm{z}_{\sob{2}{q}(\widehat K)}^q,
    \end{split}
  \end{equation}
  which, together with \eqref{eq:split_error} and \eqref{eq:interp_part} completes the proof.
\end{proof}

\begin{remark}
  Note the interpolant we construct would not be well defined in three
  spatial dimensions, $n=3$, for the range of values of $q$ considered here, 
  as the construction requires point values which, in turn, requires $z\in\sob{2}{s}(\W)$ for $2s > n$.
\end{remark}

\subsection{Bilateral approximation} \label{sec:bilateral}

There is much less in the literature on the design of two sided
approximations. The only arguments date back to Mosco and Strang
\cite{MoscoStrang:1974,Strang:1974,Strang:1974b} where the authors
consider piecewise linear approximations for dimension $n=1,2,3$ and
$q=2$. Unfortunately these arguments do not appear to extend to the
case at hand.

The difficulty in this problem can bee seen by examining Figure
\ref{fig:lagmod} around where $z\sim 0$. To keep a bilateral
constraint on the approximation, any appropriate function is
\emph{squeezed} and the only mechanism is to force it to zero on a
surrounding patch.

\subsection{Optimal approximation over \(\mathcal{R}^1(K)\)}

We modify $\interp$ (defined in equation \eqref{eq:nonlin}) in the following way at the degrees of freedom:
\begin{equation}
  \binterp(z)(x_i) =
  \begin{cases}
    \interp(z)(x_i) \text{ if }\interp(z) \geq 0 \text{ on } \widehat x_i
    \\
    0 \text{ otherwise,}
  \end{cases}
\end{equation}
as illustrated in Figure \ref{fig:bilat}.

\begin{lemma}\label{prop:one_sided_optimal_approx}
  Let $ 0 \leq z\in\sob 2q(\W)$, $q\in (1,\tfrac 43)$ and define
  \begin{equation}
    \cZ_h := \{ z_h \in \fes : 0 \leq z_h \leq z \}.
  \end{equation}
  Then $\binterp(z) \in \cZ_h$. Further, there exists a constant $C_b > 0$ such that for all elements $K\in\T{}$
  \begin{equation}
    \Norm{z - \binterp(z)}_{\leb{q}(K)}
    \leq
    C_{b} h^2 \norm{z}_{\sob{2}{q}(\widehat K)}.
  \end{equation}
\end{lemma}

  \begin{proof}
    To begin, notice that due to the regularity of $z$ we can find a
    constant $C$ independent of $h$ such that
    \begin{equation}
      \Norm{z - \binterp(z)}_{\leb{q}(K)}
      \leq
      \Norm{z - \interp(z)}_{\leb{q}(K)}
      +
      C \Norm{S}_{\leb{q}(K)},
    \end{equation}
    where
    \begin{equation}
      S
      =
      \max\qp{-\interp(z), 0}.
    \end{equation}
    Control of the first term is given in Theorem
    \ref{the:optimalapprox}. For the second, notice that $S$ only has
    support when $\cI(z) > z$ in a vicinity of when $z$ vanishes. This
    can only happen if $z$ is locally convex. Further, since $\cI(z)
    > 0$ whenever $z>0$ we have that
    \begin{equation}
      \norm{\max\qp{-\interp(z),0}}
      \leq
      \norm{R}
    \end{equation}
    and hence
    \begin{equation}
      \Norm{S}_{\leb{q}(K)}
      \leq
      \Norm{R}_{\leb{q}(K)}
      \leq
      C h^2 \norm{z}_{\sob{2}{q}(\widehat K)},
    \end{equation}
    concluding the proof.
  \end{proof}

  \begin{remark}
  Two sided bounds are conjectured in \cite{MoscoStrang:1974} for
  higher dimension, at least with $q=2$, although their results only
  seem to hold with dimension $n\leq 3$.  The conjecture we make is
  that optimal bilateral approximations in $\leb{q}(\W)$ are only
  possible when $2-\tfrac n q > 0$, hence we believe our operator is a
  constructive example of that studied in \cite{MoscoStrang:1974},
  i.e., is valid for $n=3$ and $q=2$.
\end{remark}

\section{Numerical Examples}
\label{sec:num}

In this section, we present numerical results to demonstrate the
effectiveness of the error estimate and adaptive routine against an
exact solution within the framework of \S\ref{sec:setup} and
\ref{sec:apost}, that is, convex $\Omega \subseteq \mathbb{R}^2$ with
polygonal boundary.  We then present some more challenging situations
in which the theory of the preceding sections may fail, but in which the
error estimate may still prove useful.  In particular, we test the
estimate on a non-convex domain in $\mathbb{R}^2$ with a re-entrant
corner as well as a three-dimensional example.

All simulations presented here are conducted using \texttt{deal.II},
an open source C\texttt{++} software library providing tools for
adaptive finite element computations \cite{dealII91}.  We note here
that \texttt{deal.II} uses quadrilateral meshes.  Since all coarse
meshes are uniform meshes consisting of squares, and refinement is by
quadrisection, regularity of the mesh does not degrade under heavy
refinement.  At each stage of the iterative process used to solve the
variational inequality, a preconditioned conjugate gradient method is
used to solve the algebraic problem.  As the mesh is refined, hanging
nodes are created, and are constrained so that the resulting numerical
solution is continuous.

\subsection{Example 1}\label{sec:l4:ex_1}

For our first example, we choose the exact solution used by the authors of
\cite{christof1754finite}.
Let $r, \theta$ denote polar coordinates centred at $(0.5, 0)$, that is

\begin{equation}
r(x,y) = ((x-0.5)^2 + y^2)^{1 \slash 2},
\end{equation}

\begin{equation}
\theta(x,y) = \arccos\left(\frac{x - 0.5}{r}\right).
\end{equation}
We define the function

\begin{equation}
\tilde{u}(r, \theta) : = - r ^{3 \slash 2} \sin (\tfrac{3}{2} \theta),
\end{equation}
and define $\psi$ to be a ninth order spline defined by endpoint values and gradients as follows (see also \cite[\S 6]{christof1754finite}).
We impose the following values to determine $\psi$.

\begin{equation}
\psi(0) =1, \,\,  \psi '(0) = ... = \psi ''''(0) = \psi(0.45) = \psi '(0.45) = ... = \psi ''''(0.45) = 0,
\end{equation}
and $\psi(s) = 0$ for any $s < 0 $ or $s \geq 0.45$.  Then if we
choose $f$ accordingly, $10 \psi \tilde{u}$ solves
\eqref{eq:weak_form}.  {We also point out that this
  exact solution satisfies \hyperref[struct_of_contact_set]{Condition
    $(\mathbf{A})$}, since by the definition its contact set is}

\begin{equation*}
  {\partial \W \backslash \{(x,0) : 0.05\leq x \leq 0.5\},}
\end{equation*}
{that is, a single connected component consisting of an open interval
in $\partial \W$.} All numerical approximations satisfy
\hyperref[def:condition_A_h]{Condition $(\mathbf a_h)$}. An example of the
verification (see Remark \ref{rem:verification_of_condition}) of this condition is shown in Figure \ref{fig:plot_contact_set},
where the discrete contact set for Example 1 is shown for varying levels of mesh
refinement.

The problem is initialised on a coarse uniform mesh with $h = 1 \slash
8$.  For the adaptive algorithm we use D\"orfler marking (see
\cite{Dorf1996}) with refine fraction $\beta = 0.9$ and no coarsening.
A numerical approximation to problem \eqref{eq:weak_form} with datum
$f$ chosen so that $10 \psi \tilde{u}$ is the exact solution is shown
in Figure \ref{eq:sol_1}, represented on the final mesh
$\mathcal{T}^{15}$ produced by the adaptive algorithm, consisting of
around 80,000 degrees of freedom.  The progression of the
$\leb{4}(\W)$-error, $\Norm{u - U}_{\leb{4}(\W)}$, and error estimate
$\eta^{1 \slash 4}$ as defined in equation \eqref{eq:l4:pos_bound}
under adaptive mesh refinement is shown in Figure
\ref{fig:strategy_comparison}.  For a given number of degrees of
freedom, the adaptive algorithm reduces the error and appears to give
a slightly better convergence rate than uniformly refining the mesh in
terms of degrees of freedom.  We observe that the over-estimation
factor (or effectivity index as it is often called) of the error
estimate, is approximately 20.  As expected from the theory, this
factor is approximately constant once the asymptotic regime is
reached.  Comparing with Figure \ref{eq:adaptive_uniform_1}, we see
that asymptotically, error is lower under adaptive refinement than
uniform in terms of degrees of freedom, although both are optimal.

A selection of adaptive meshes are shown in Figures \ref{ex_1_mesh1}-\ref{ex_1_mesh4}.
These meshes show refinement around the main features of the solution.
In particular, the mesh is heavily refined around large gradients, along with significant refinement where the boundary conditions change type and constraints are active.

Finally, we demonstrate the ability of adaptive mesh refinement, driven by our error estimate, to compensate for suboptimal convergence in $\leb{p}(\W)$ for large $p$, observed in \cite[Table 2]{christof1754finite}. 
We set $p = 32$ and adapt the mesh using D\"orfler marking with refine fraction $\beta = 0.8$ and no coarsening.
In Figure \ref{fig:l32_rates}, we see that the error under uniform mesh refinement is asymptotically suboptimal, but that optimal rates in terms of number of degrees of freedom is achieved with the adaptive scheme. 
Spatial distribution of the error estimate is shown in Figure \ref{fig:l32_estimate}.


\begin{figure}[ht!]
    \centering 
    \begin{subfigure}{0.32\textwidth}
        \includegraphics[width=\linewidth]{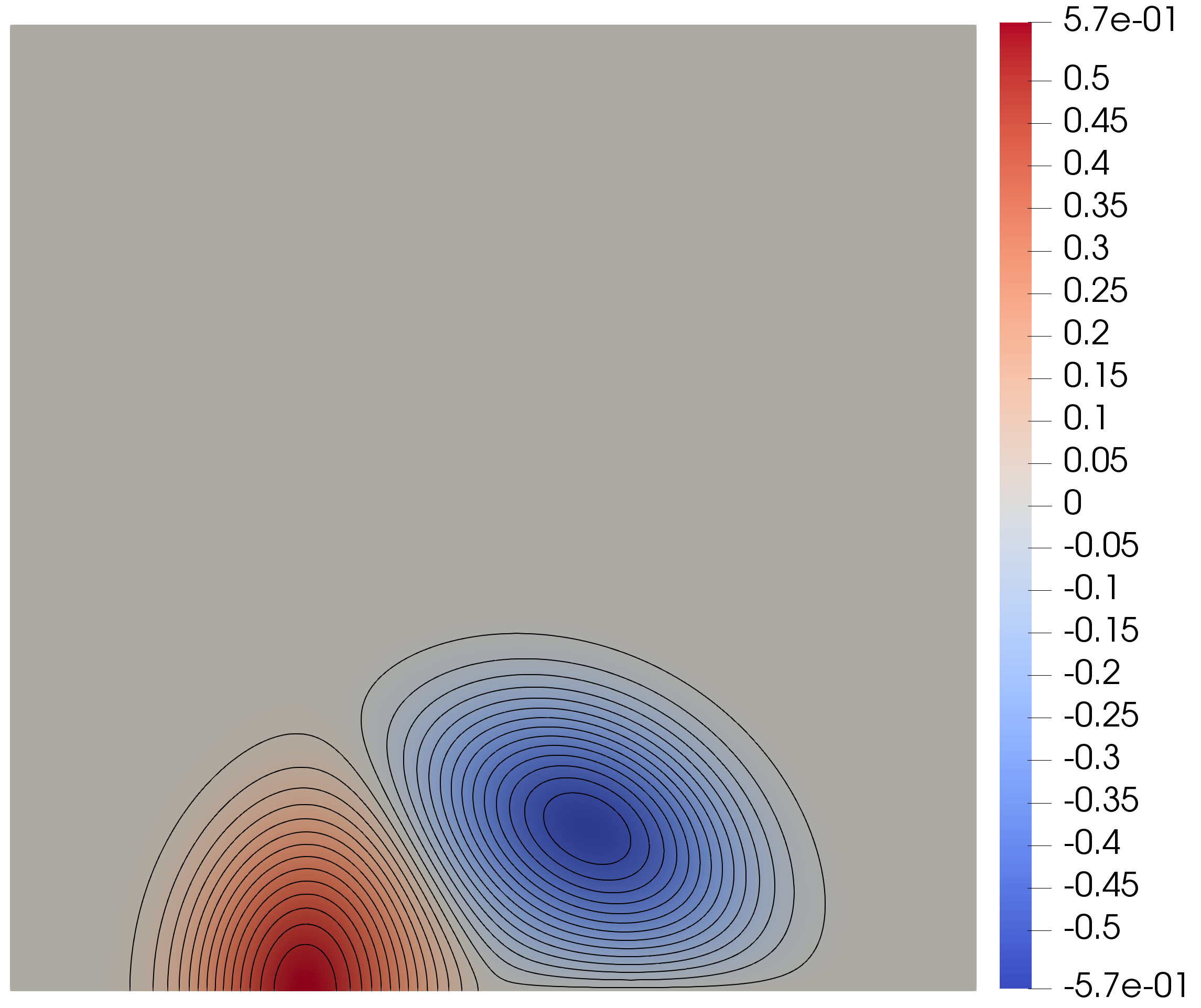}
        \caption{Contour plot of adaptive approximation.
            \label{eq:sol_1}
        }
    \end{subfigure}\hfil 
    \begin{subfigure}{0.32\textwidth}
        \includegraphics[width=\linewidth]{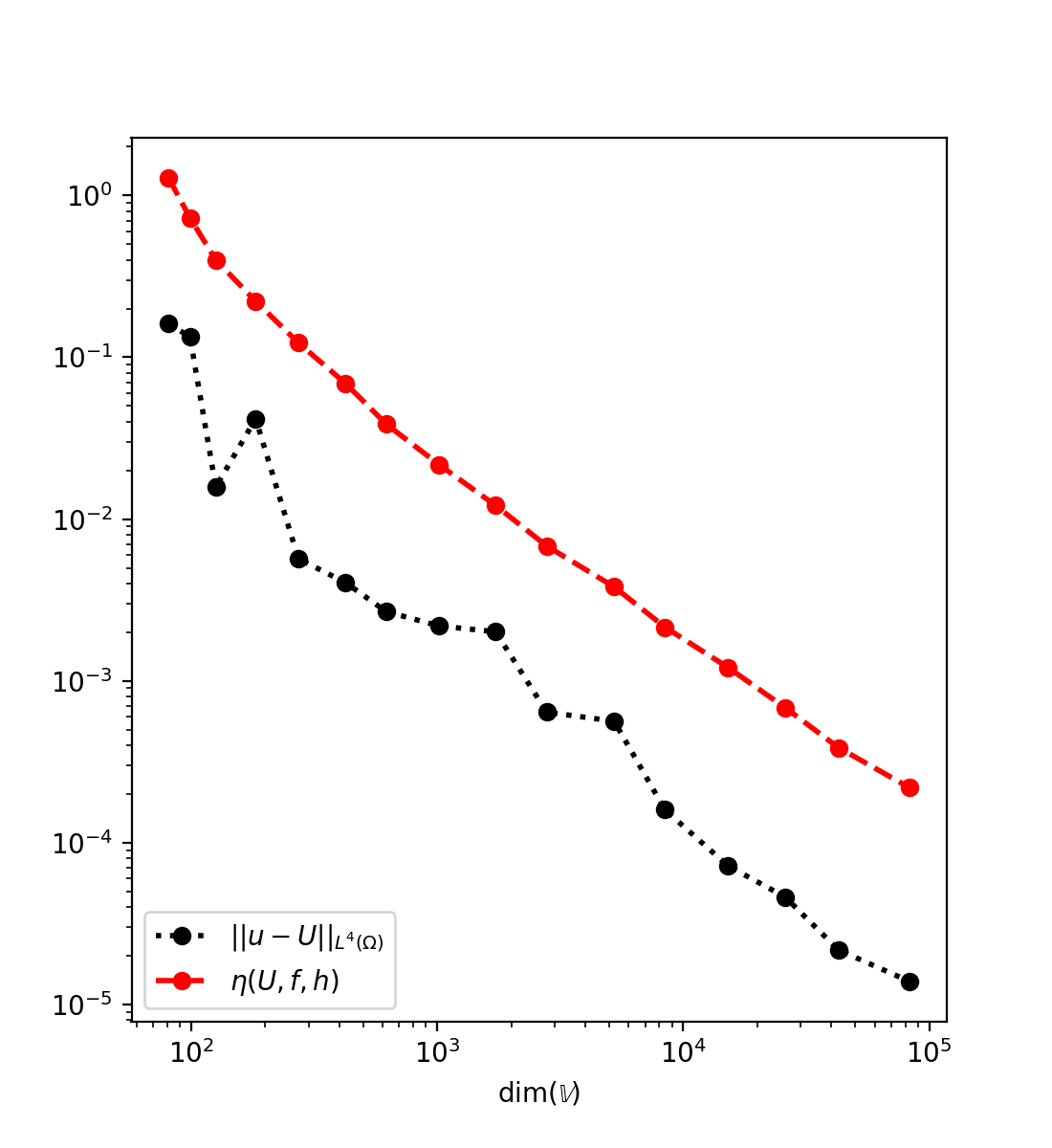}
        \caption{Behaviour of error estimate.
            \label{fig:strategy_comparison}
        }
    \end{subfigure}\hfil 
    \\
    \begin{subfigure}{0.24\textwidth}
        \includegraphics[width=\linewidth]{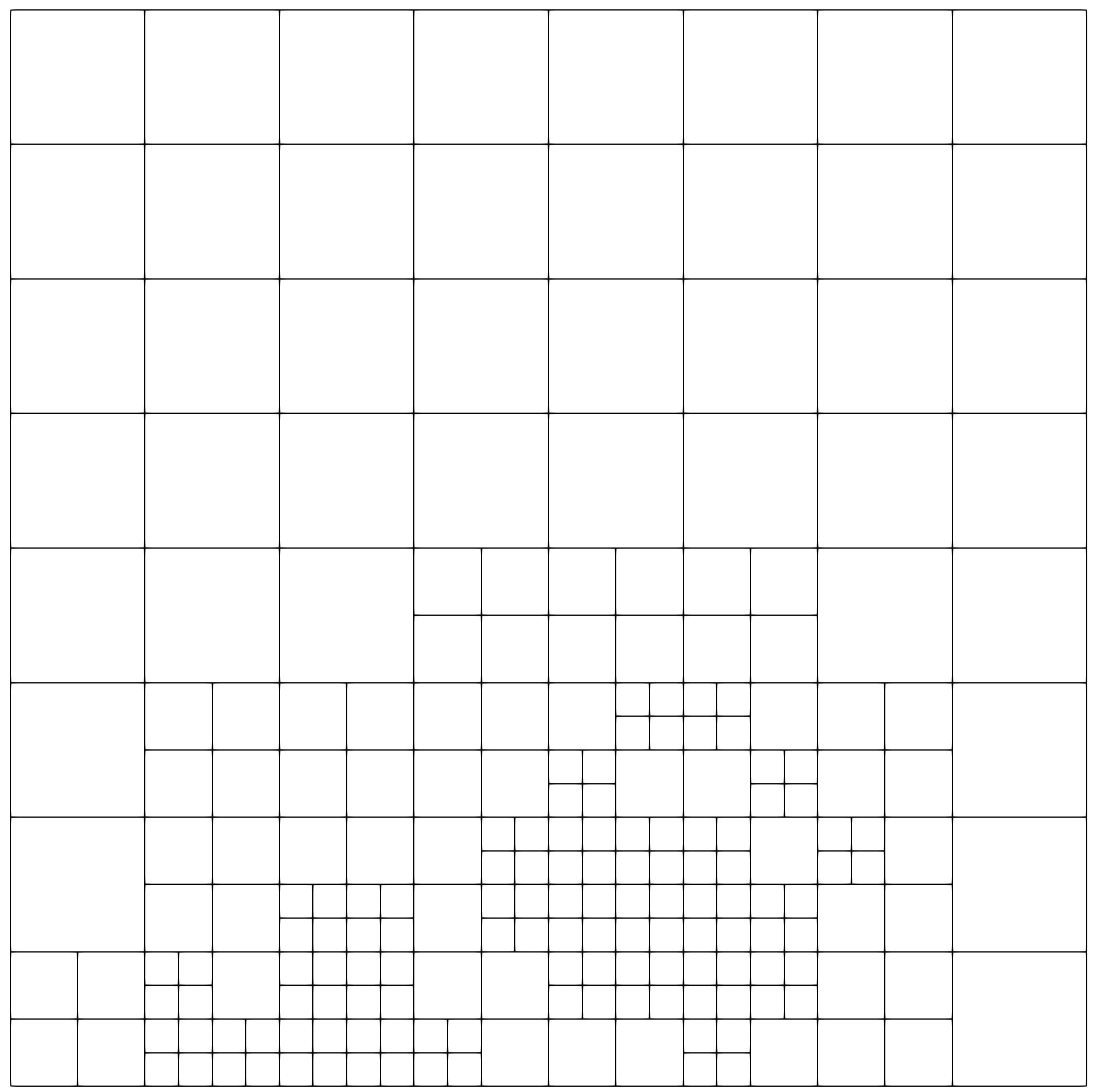}
        \caption{$\mathcal{T}^4$
            \label{ex_1_mesh1}}
    \end{subfigure}\hfil 
    \begin{subfigure}{0.24\textwidth}
        \includegraphics[width=\linewidth]{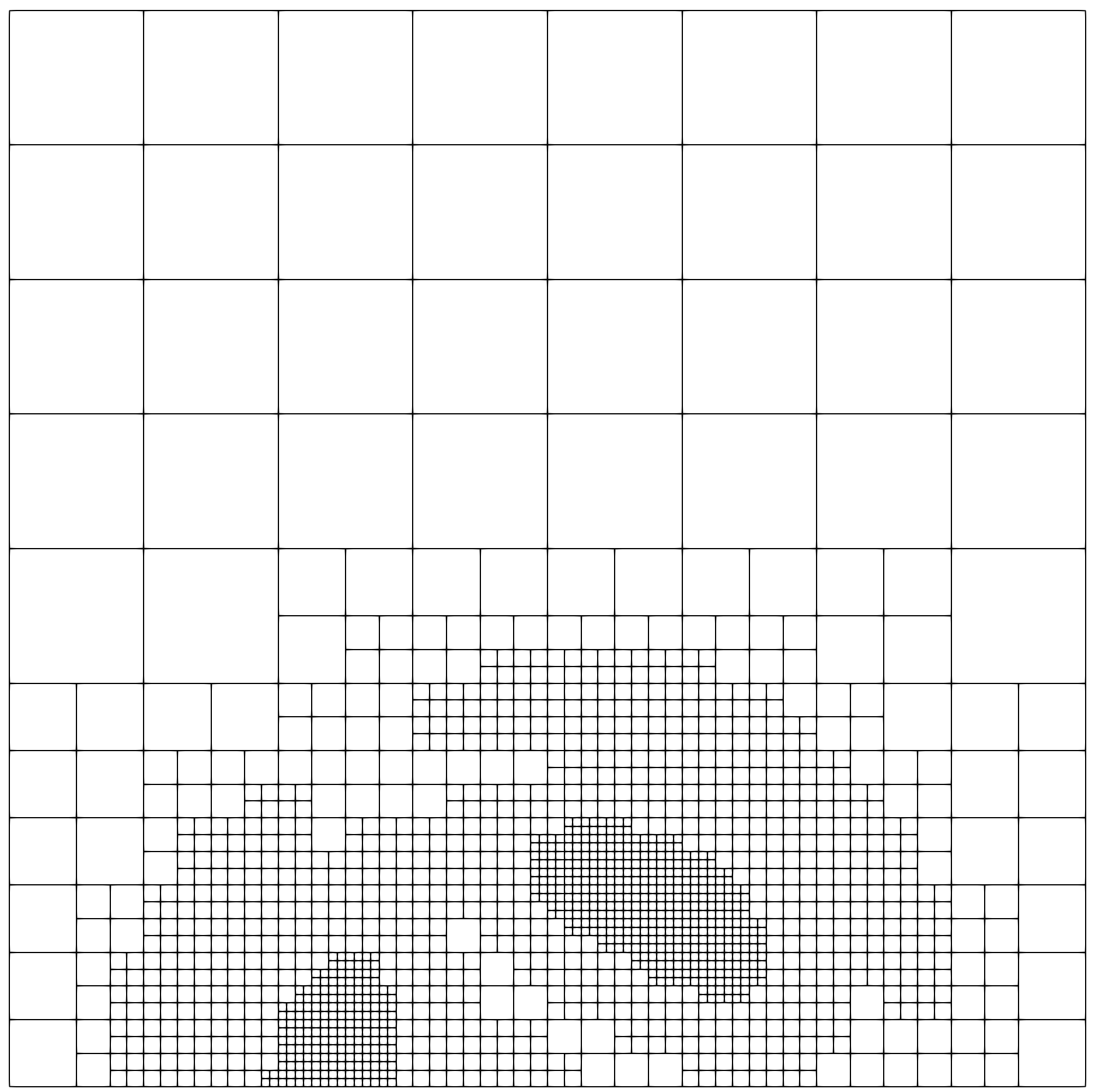}
        \caption{$\mathcal{T}^8$}
    \end{subfigure}\hfil 
    \begin{subfigure}{0.24\textwidth}
        \includegraphics[width=\linewidth]{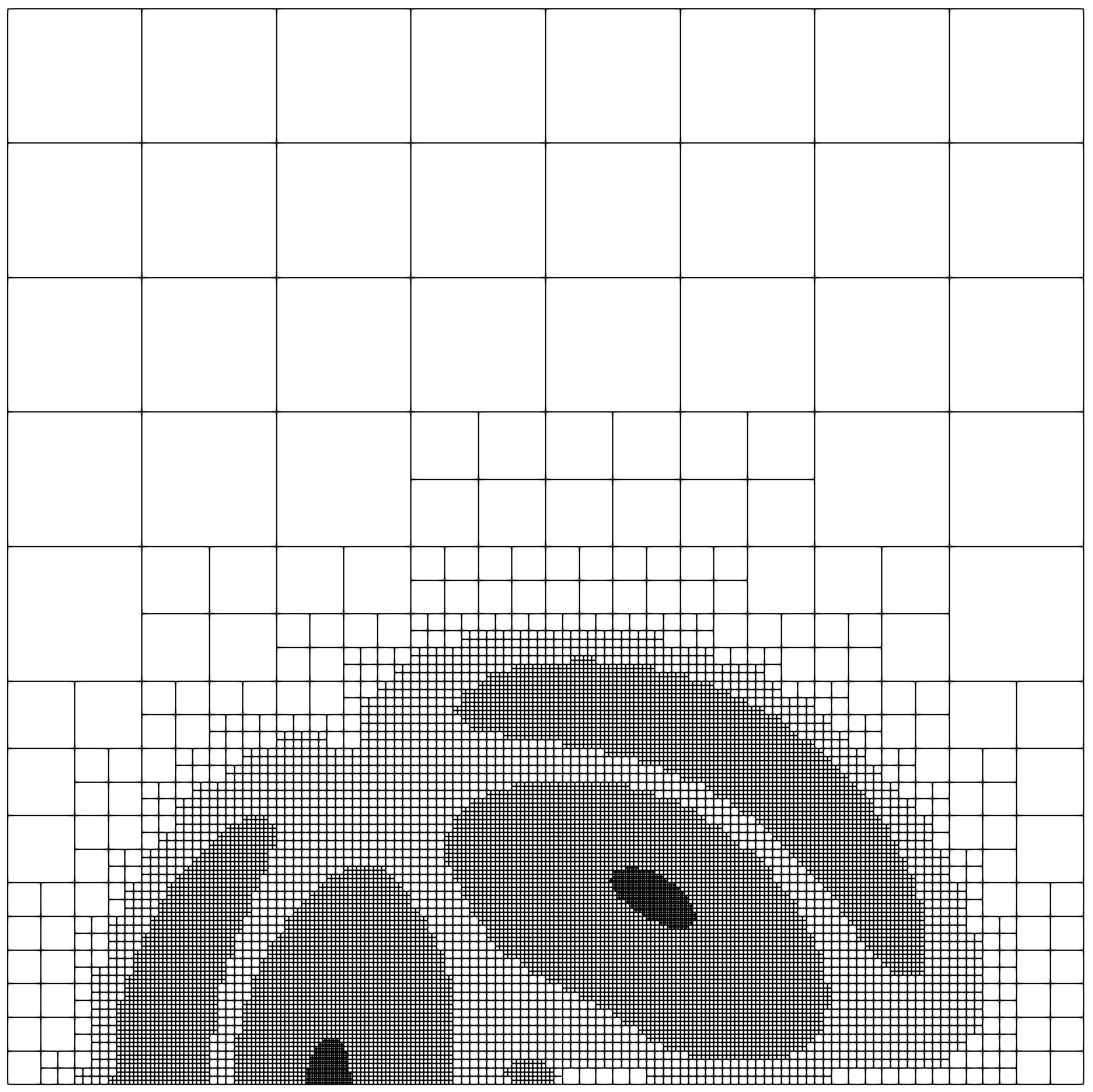}
        \caption{$\mathcal{T}^{12}$}
    \end{subfigure}\hfil 
    \begin{subfigure}{0.24\textwidth}
        \includegraphics[width=\linewidth]{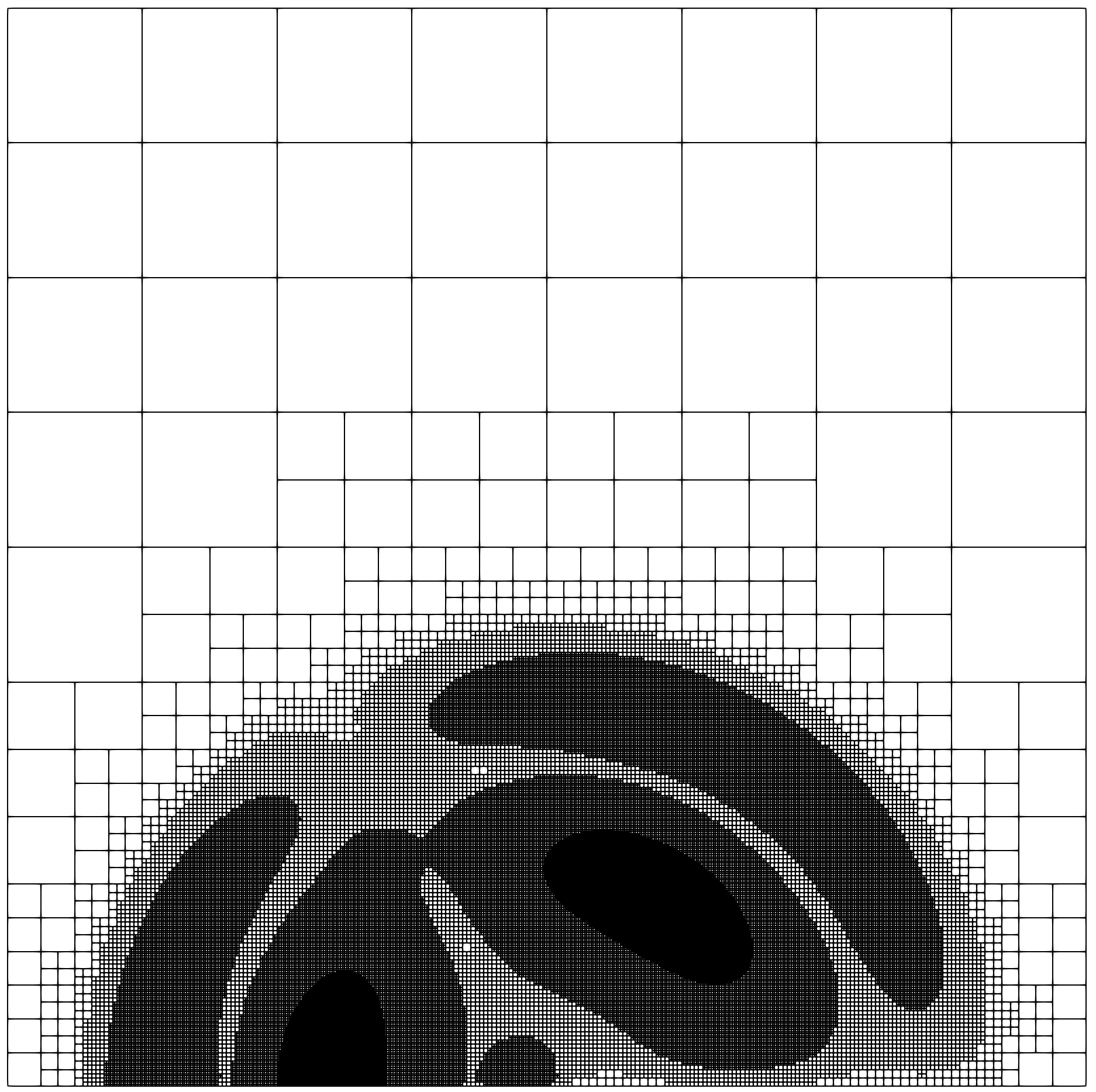}
        \caption{$\mathcal{T}^{15}$\label{ex_1_mesh4}}
    \end{subfigure}\hfil
    \caption{{Example 1 \S\ref{sec:l4:ex_1}, contour plot
        and various iterations of the adaptive mesh $\mathcal{T}^i$ of
        an approximation to a function $u \in \sobh{2}(\W)$. Figure
        \ref{fig:strategy_comparison} shows a double-logarithmic plot
        comparing the error in $\leb{4}(\W)$ (black line) with the
        error estimate (red line).}}
\end{figure}

\begin{figure}[ht!]
  \centering
  \begin{subfigure}{0.35\textwidth}
      \includegraphics[width=\linewidth]{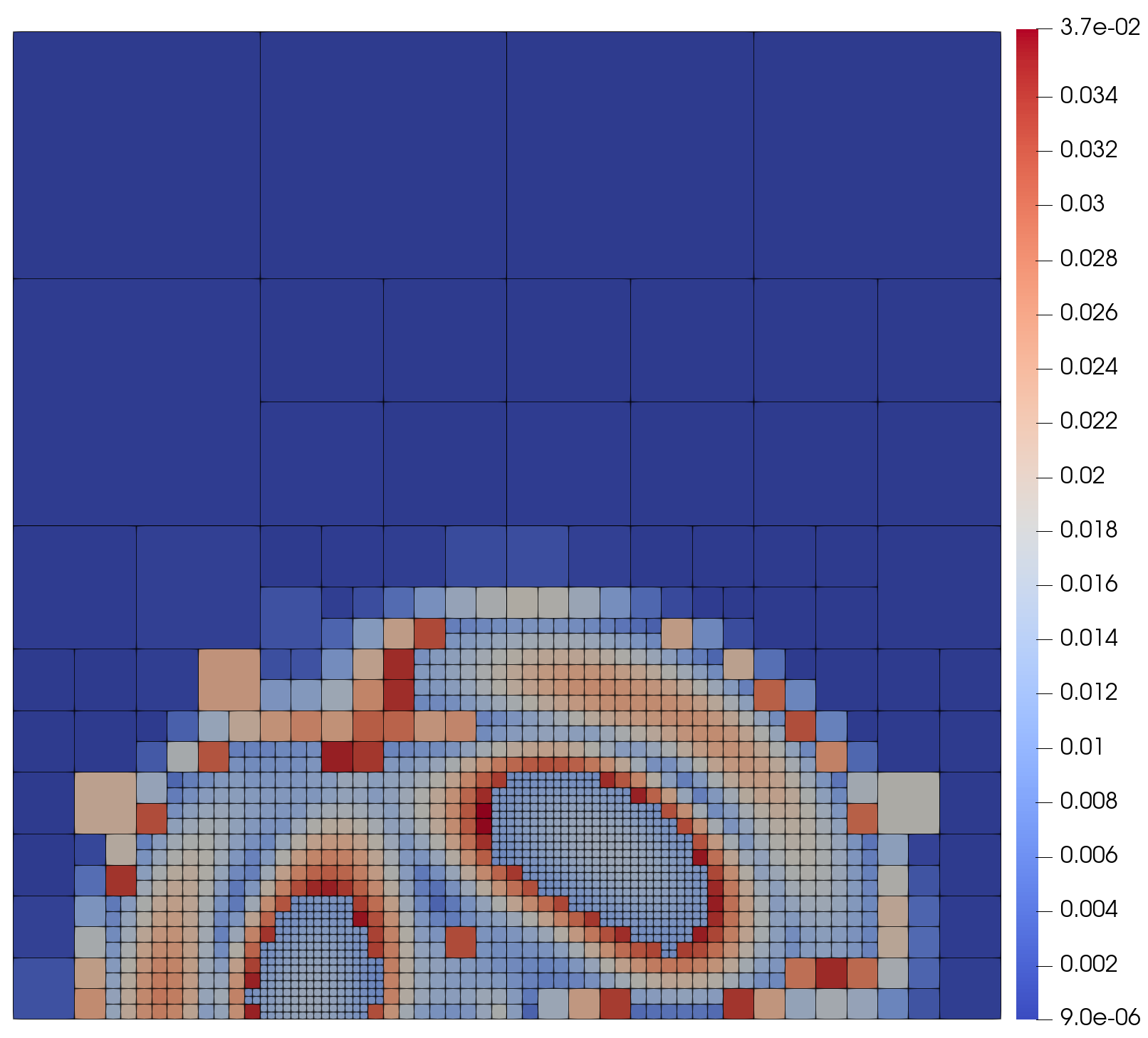}
      \caption{Example 1. Spatial distribution of contributions to $\leb{32}(\W)$ error estimate .}
     \label{fig:l32_estimate}
   \end{subfigure}\hfil
   \begin{subfigure}{0.35\textwidth}
    \includegraphics[width=\linewidth]{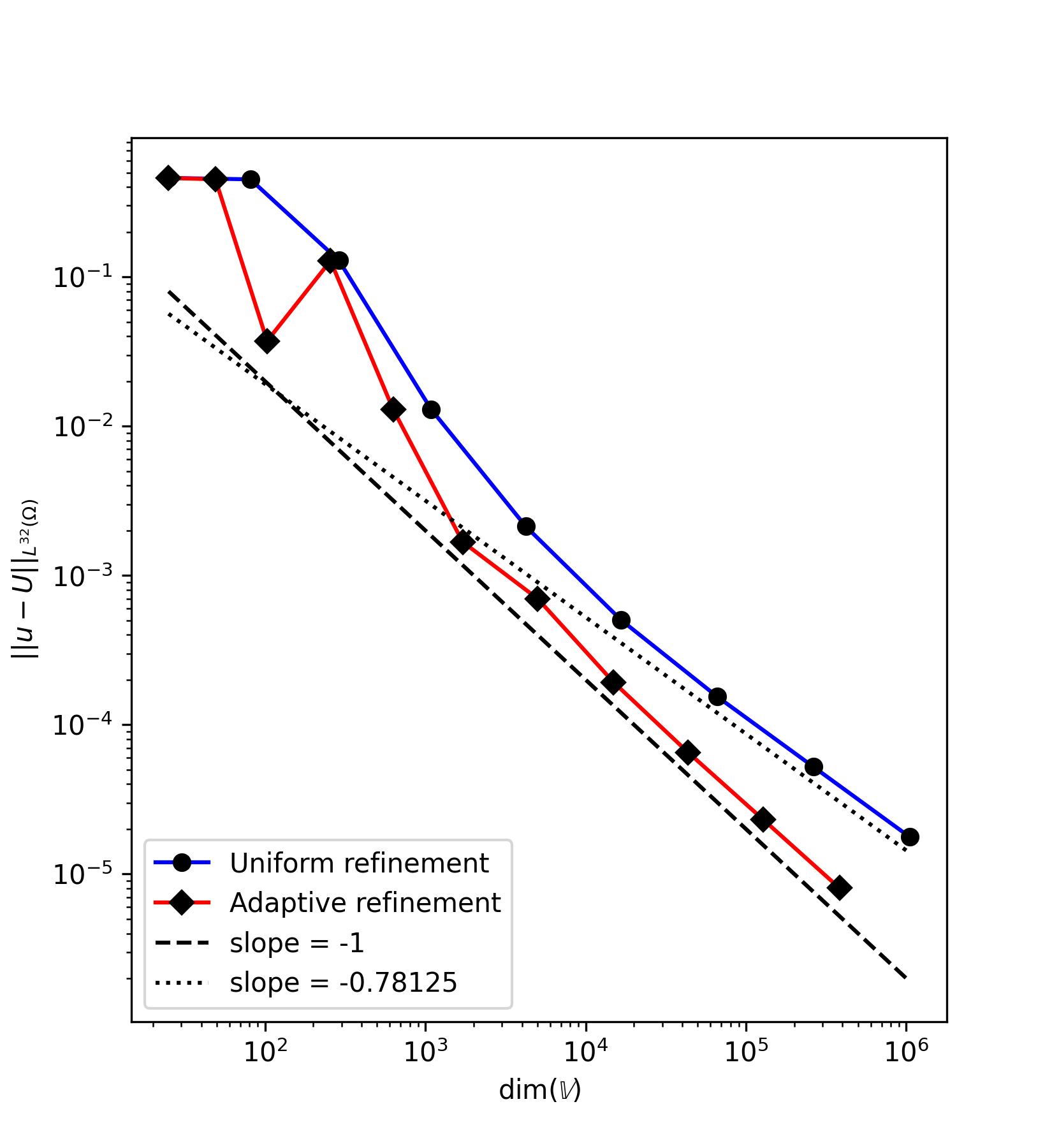}
    \caption{Example 1. Comparison of adaptive and uniform refinement.}
    \label{fig:l32_rates}
  \end{subfigure}
   \caption{Comparison of
   $\leb{32}(\W)$ error computed against exact solution for uniform
   and adaptive meshes for example 1.  Uniform mesh refinement
   delivers suboptimal convergence of order
   $h^{\tfrac{25}{16}-\varepsilon}$ as predicted by the a priori
   results of \cite{christof1754finite}. Optimal rates are recovered
   using adaptive mesh refinement.}
   \label{fig:l32}
\end{figure}

\begin{figure}
  \centering
  \includegraphics[width=0.32\linewidth,trim={3cm 0 3cm 0},clip]
  {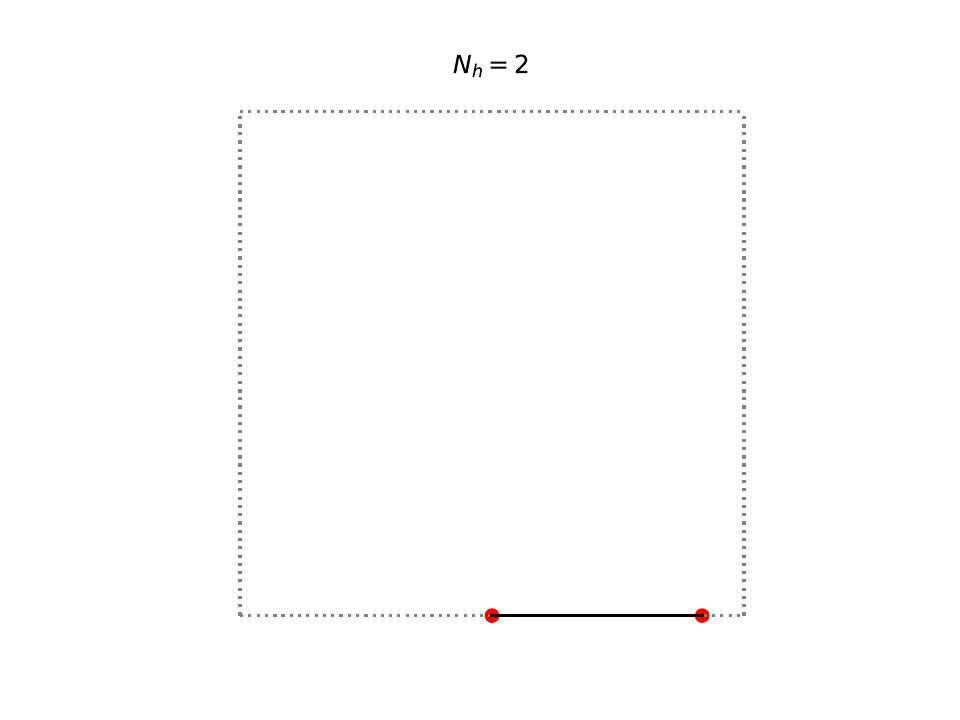}
  \includegraphics[width=0.32\linewidth,trim={3cm 0 3cm 0},clip]
  {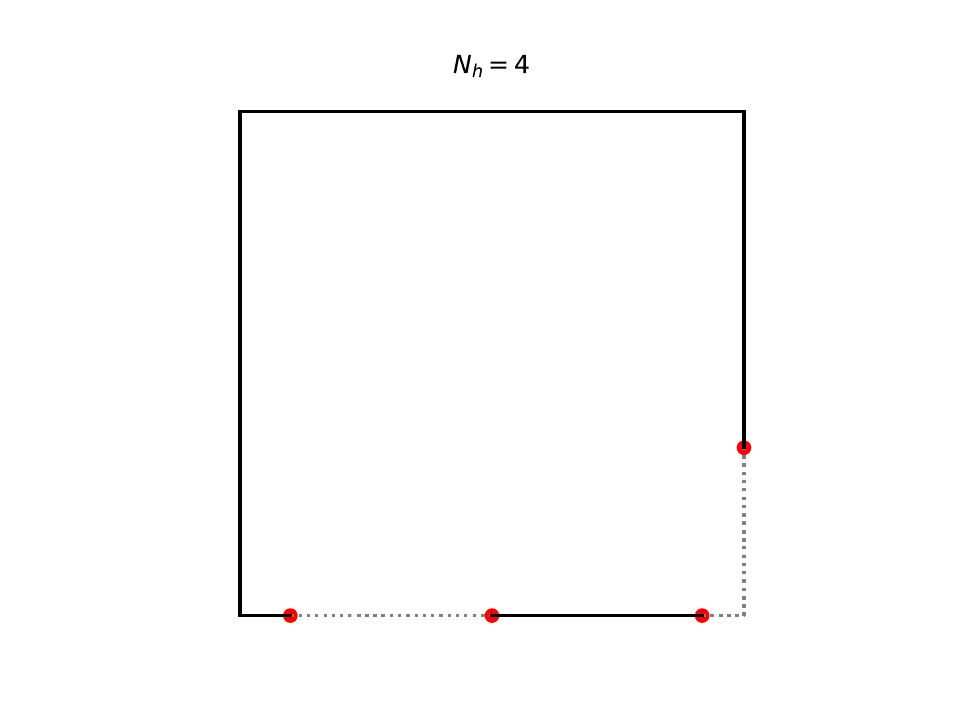}
  \includegraphics[width=0.32\linewidth,trim={3cm 0 3cm 0},clip]
  {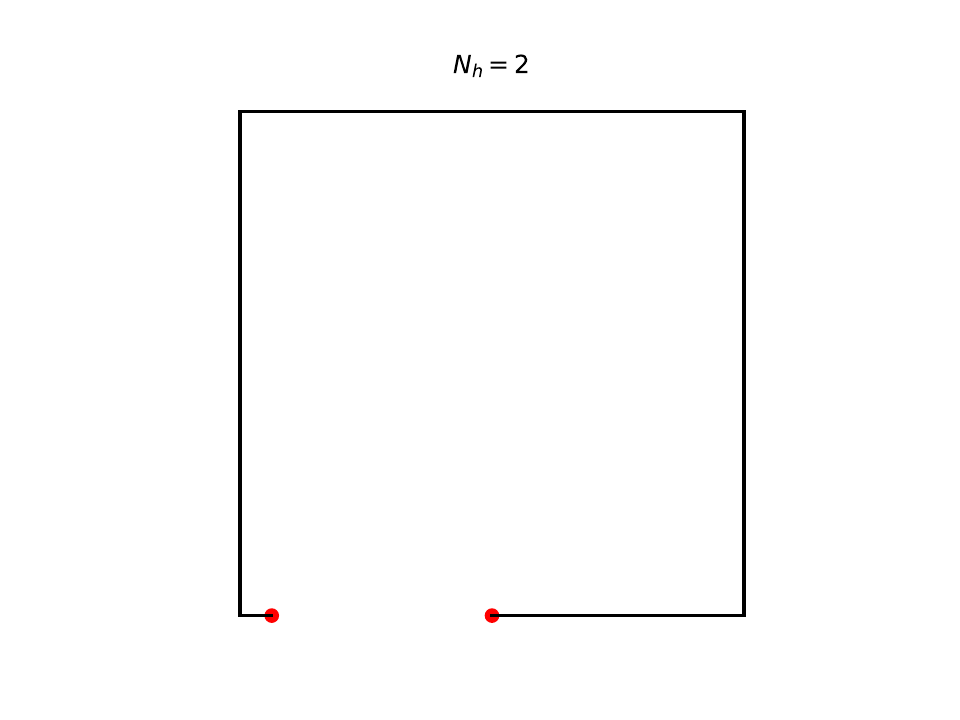}
  \caption{{Example 1 \S\ref{sec:l4:ex_1}, visualisation of the discrete contact set $\cA_U$. Meshes $\mathcal{T}^{6}$, $\mathcal{T}^{11}$, $\mathcal{T}^{15}$. Solid line is contact (i.e. $U=0$), dashed line indicates non-contact. Critical points are shown with red dots. A boundary degree of freedom is in $\cA_U$ if it is zero. For this example the discrete contact set varies on coarse meshes, and once sufficient resolution is attained we have two critical points where the boundary condition changes (note that this matches the known exact solution).
  \label{fig:plot_contact_set}}}
\end{figure}

\subsection{Example 2: Re-entrant Corner}\label{sec:l4:ex_3}

To test the estimate in the presence of a geometric singularity, we introduce a re-entrant corner to the domain.
In this example, data $f$ is chosen to ensure that both boundary constraints are active.
The problem data is selected to try and force the solution to be close to

\begin{equation}
w := \sin(2 \pi (r - b)^2) - 0.5,
\end{equation}
where $r = (x^2 + y^2)^{\frac 12}$ and $b$ can be varied to force different behaviours of the solution.
For this example we make the choice $b = 0.91$ and set $f = \Delta w + w$.

A numerical solution to this problem is shown in Figure \ref{sol_3}.
Under uniform mesh refinement, the error estimate converges to zero at a suboptimal rate.
Optimality is restored using an adaptive routine utilising D\"orfler marking with refinement fraction of 0.8 and no coarsening.
A sample of the meshes produced is given in Figures \ref{ex_3_mesh1}-\ref{ex_3_mesh4}.
The behaviour of the error estimate under uniform and adaptive refinement is shown in Figure \ref{suboptimal_convergence_2}.
The slopes of the error-reduction curves reveal suboptimal convergence in the uniform case which is somewhat recovered by adaptive mesh refinement, but to different degrees.
This time we see most refinement around gradients and features in the solution.
Interestingly, there is little refinement around the re-entrant corner where the solution is expected to lose regularity.
Instead, the error estimate prioritises resolution of the solution near where the boundary constraints are active.
This simulation was performed with the D\"orfler marking criterion with refine fraction 0.8 and no coarsening.
The algorithm was initialised on a coarse uniform mesh with $h = 1 \slash 16$.

\begin{figure}[ht!]
    \centering 
    \begin{subfigure}{0.23\textwidth}
        \includegraphics[width=\linewidth]{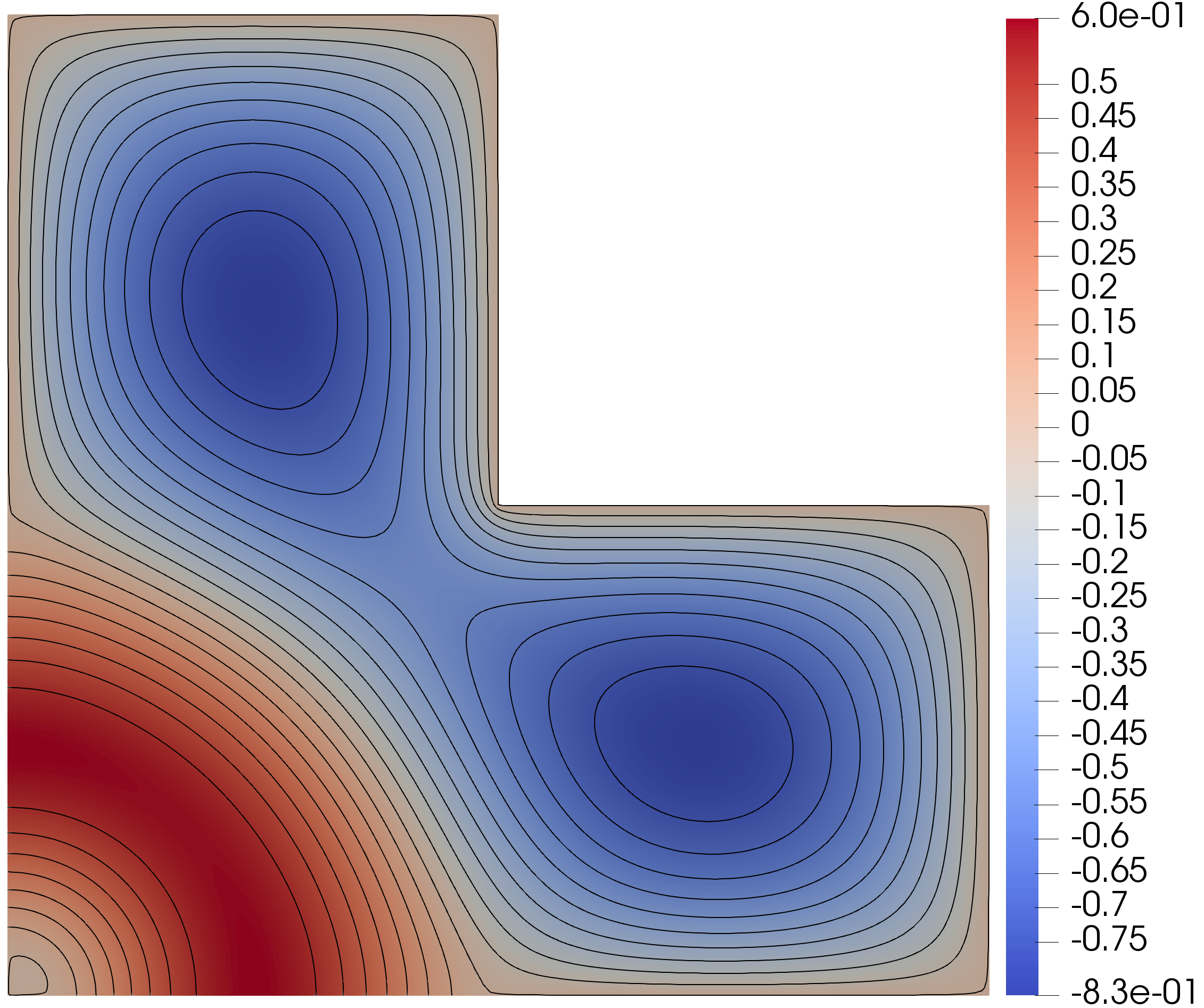}
        \caption{Contours of $U$.
            \label{sol_3}
        }
      \end{subfigure}
    \begin{subfigure}{0.19\textwidth}
        \includegraphics[width=\linewidth]{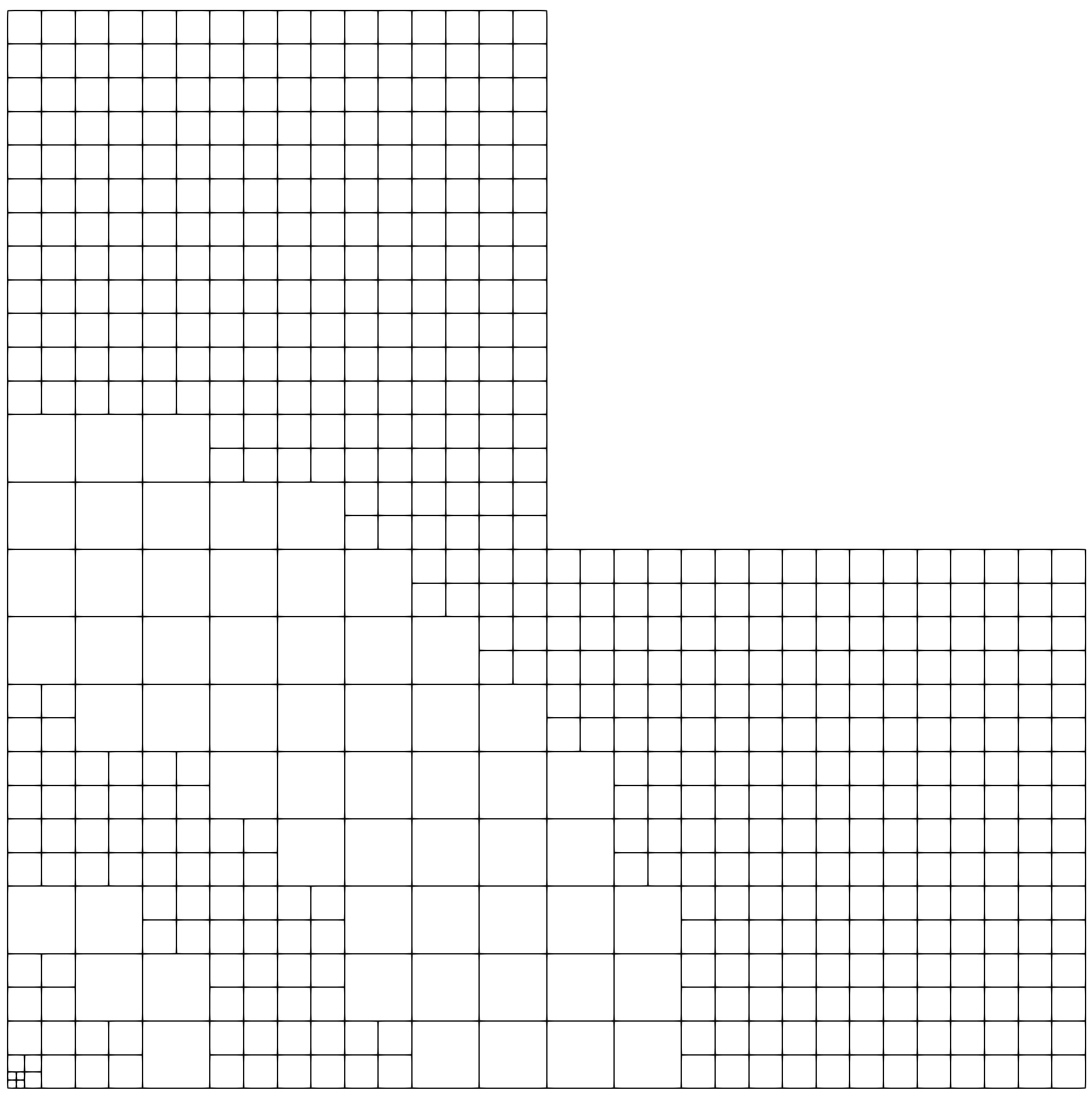}
        \caption{$\mathcal{T}^3$
            \label{ex_3_mesh1}}
    \end{subfigure}\hfil 
    \begin{subfigure}{0.19\textwidth}
        \includegraphics[width=\linewidth]{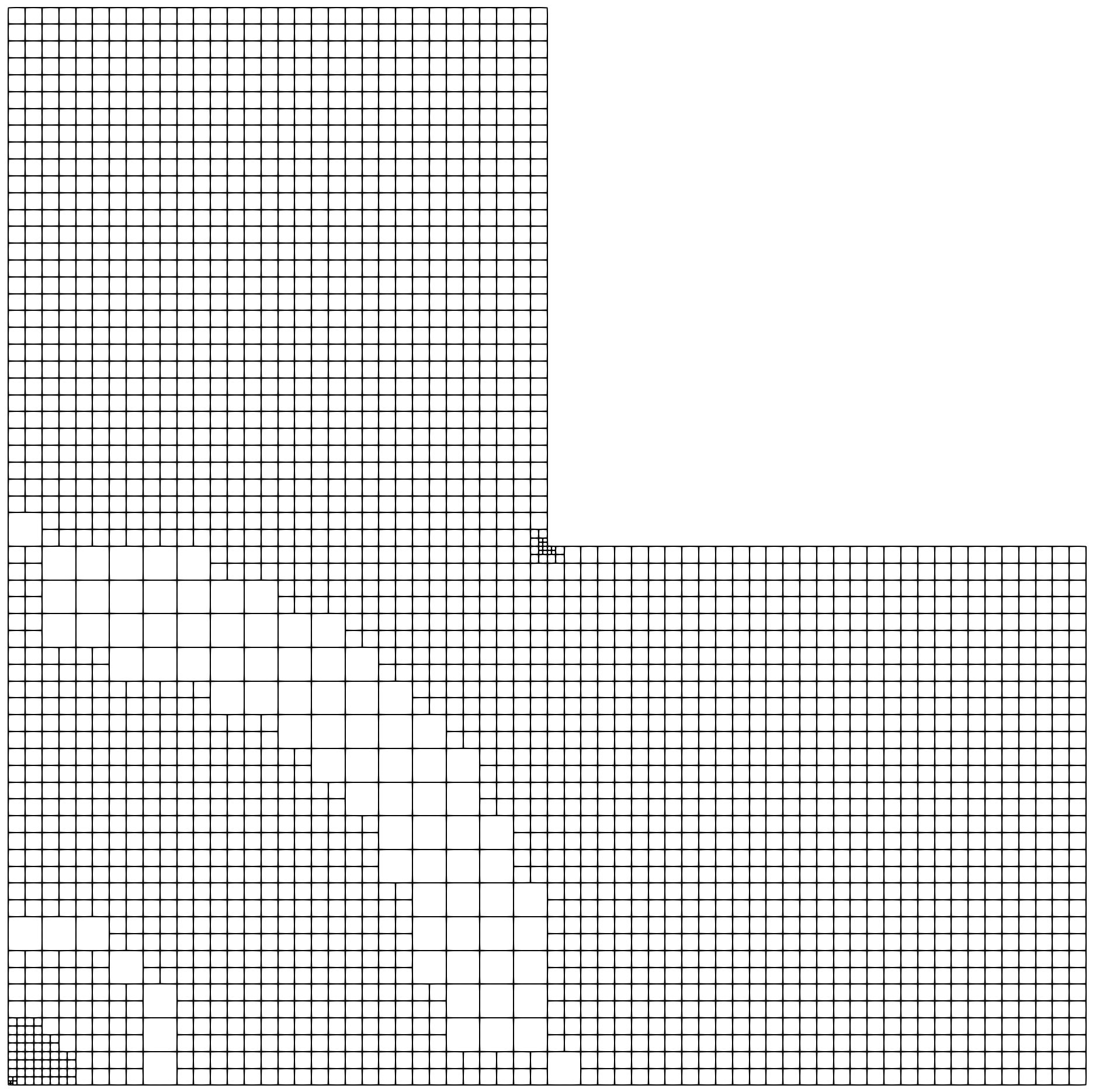}
        \caption{$\mathcal{T}^8$}
    \end{subfigure}\hfil 
    \begin{subfigure}{0.19\textwidth}
        \includegraphics[width=\linewidth]{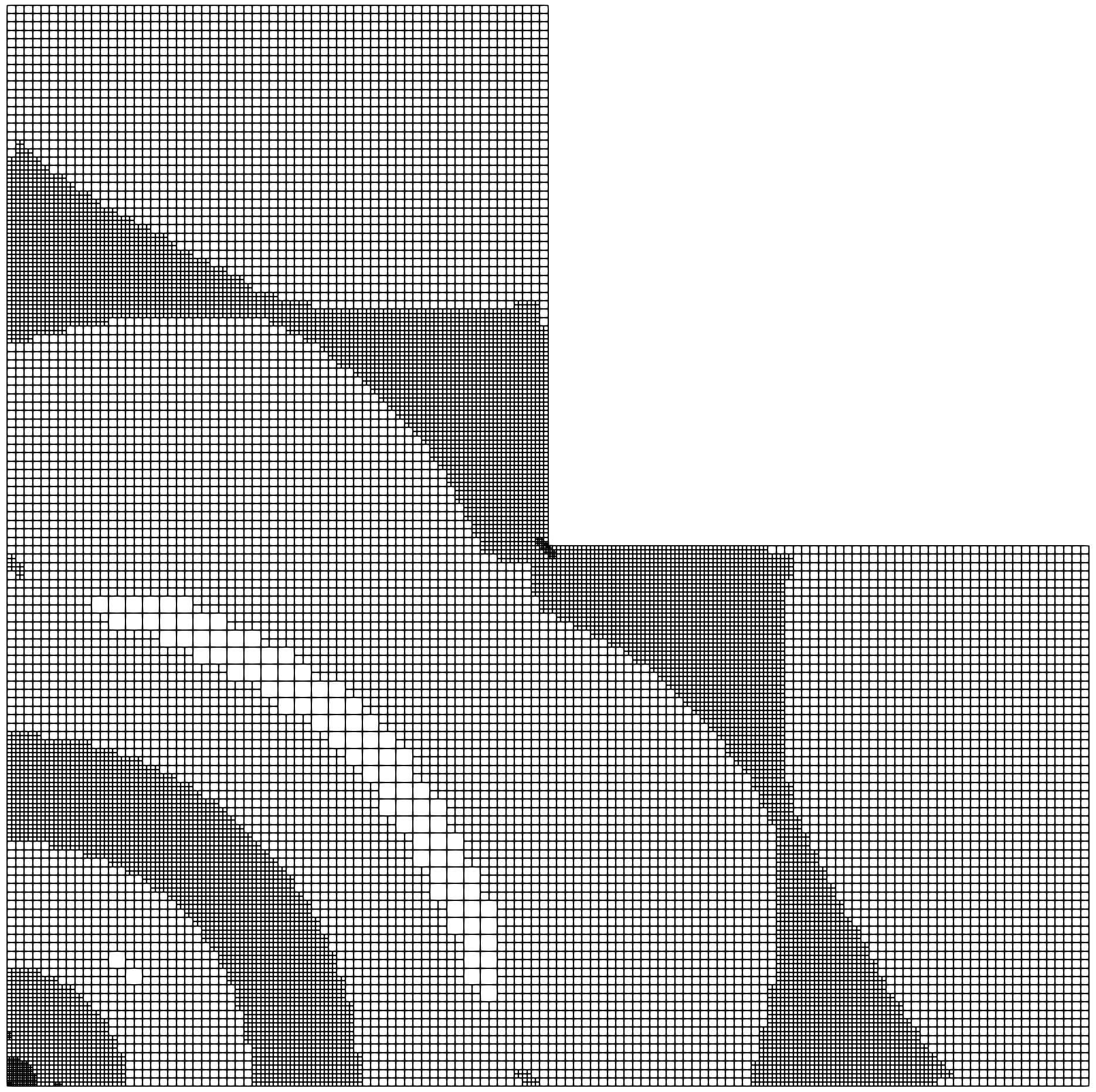}
        \caption{$\mathcal{T}^{12}$}
    \end{subfigure}\hfil 
    \begin{subfigure}{0.19\textwidth}
        \includegraphics[width=\linewidth]{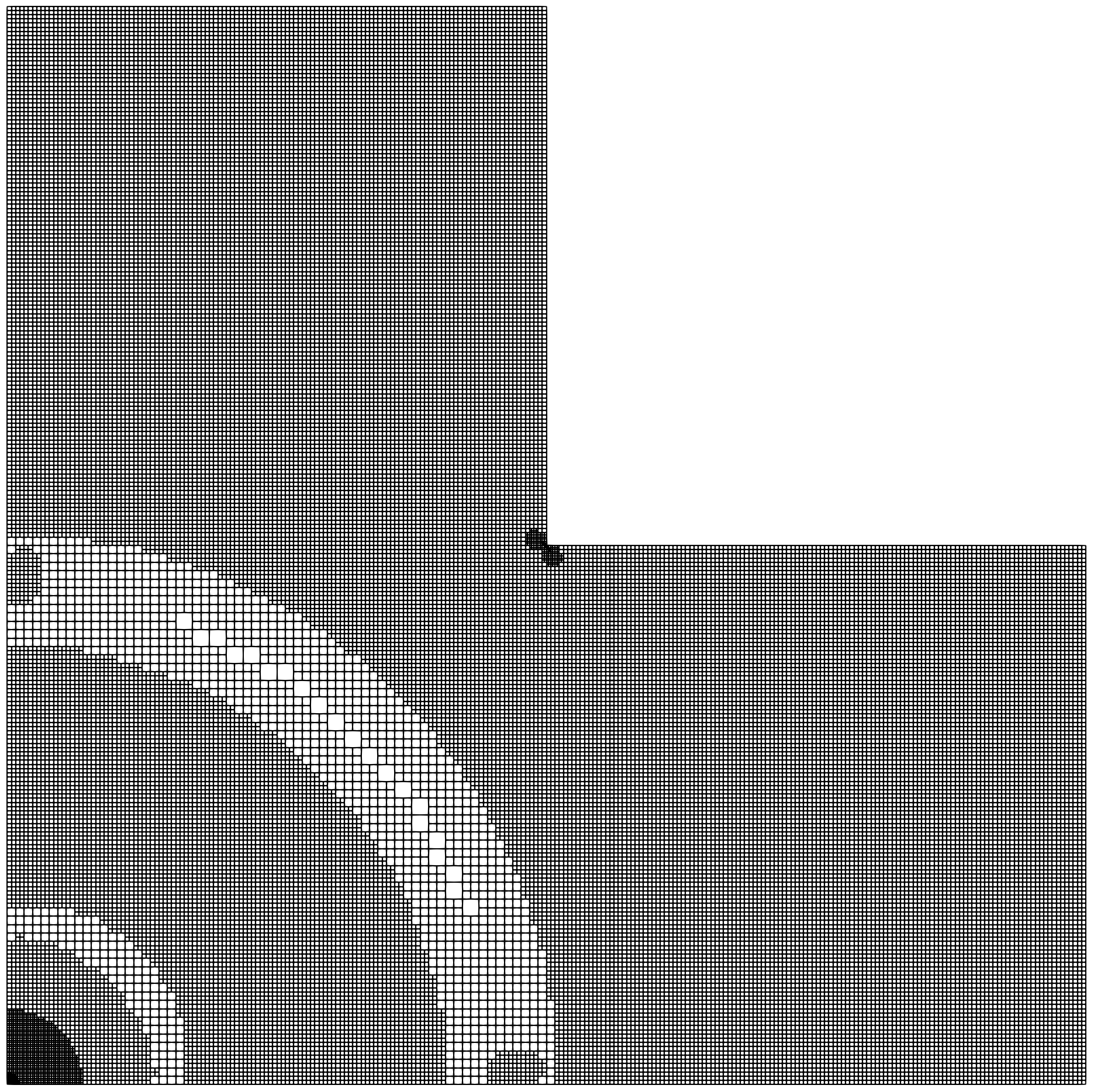}
        \caption{$\mathcal{T}^{15}$\label{ex_3_mesh4}}
    \end{subfigure}\hfil
    \caption{Example 2 \S\ref{sec:l4:ex_3}, contour plot and various iterations of the adaptive mesh $\mathcal{T}^i$ of an approximation to a function $u \in \sob{2}{(4 - \varepsilon) \slash 3}(\W) \backslash \sobh{2}(\W)$. The gradient appears to be discontinuous at the corner. In this case the mesh refines around the reentrant corner as well as along qualitative features of the solution and where the boundary conditions change type. }
\end{figure}

\begin{figure}[ht!]
  \centering
  \begin{subfigure}{0.3\textwidth}
      \includegraphics[width=\linewidth]{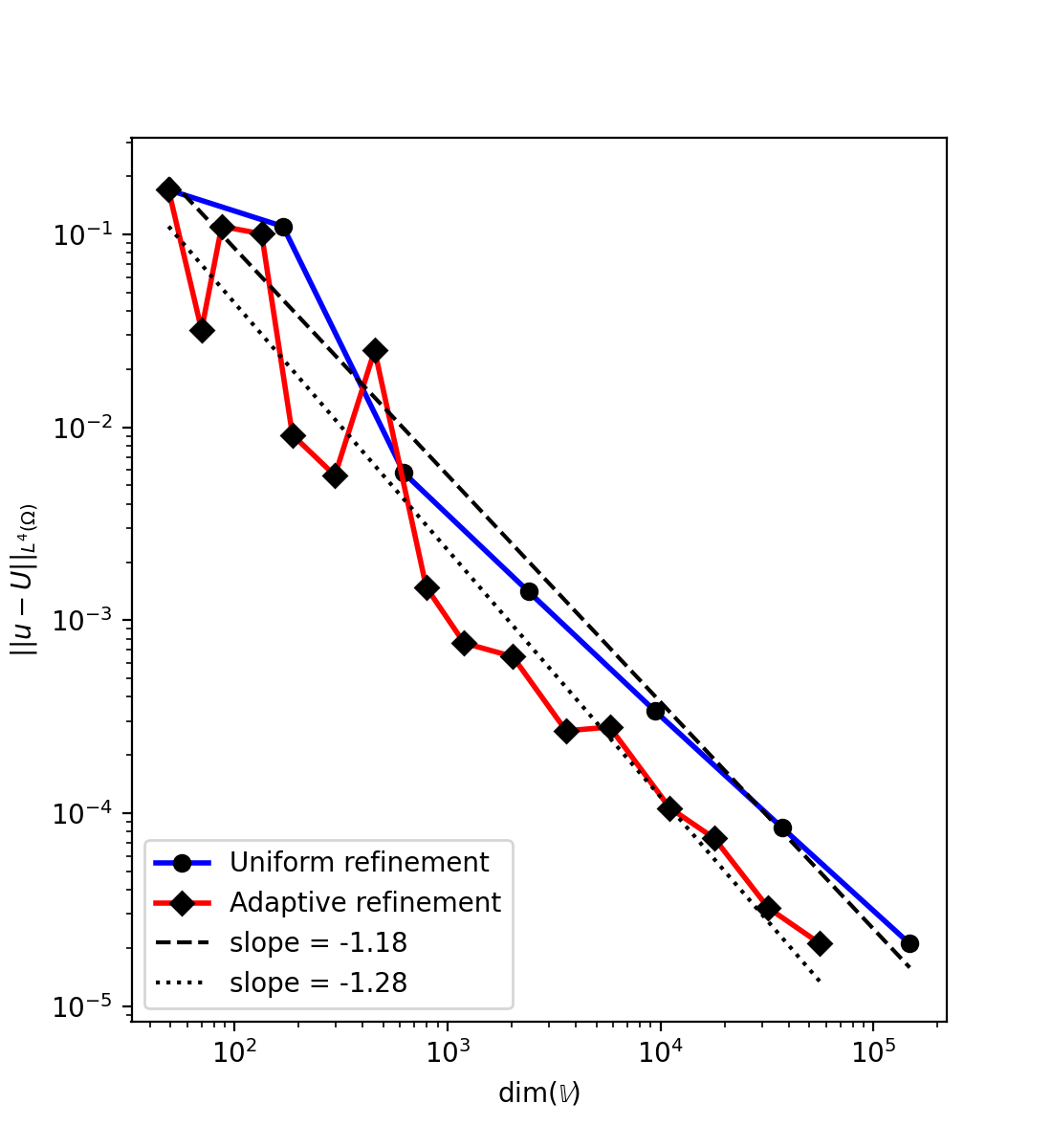}
      \caption{Example 1. Expected slope from theory is -1.}
     \label{eq:adaptive_uniform_1}
   \end{subfigure}\hfil
   \begin{subfigure}{0.3\textwidth}
    \includegraphics[width=\linewidth]{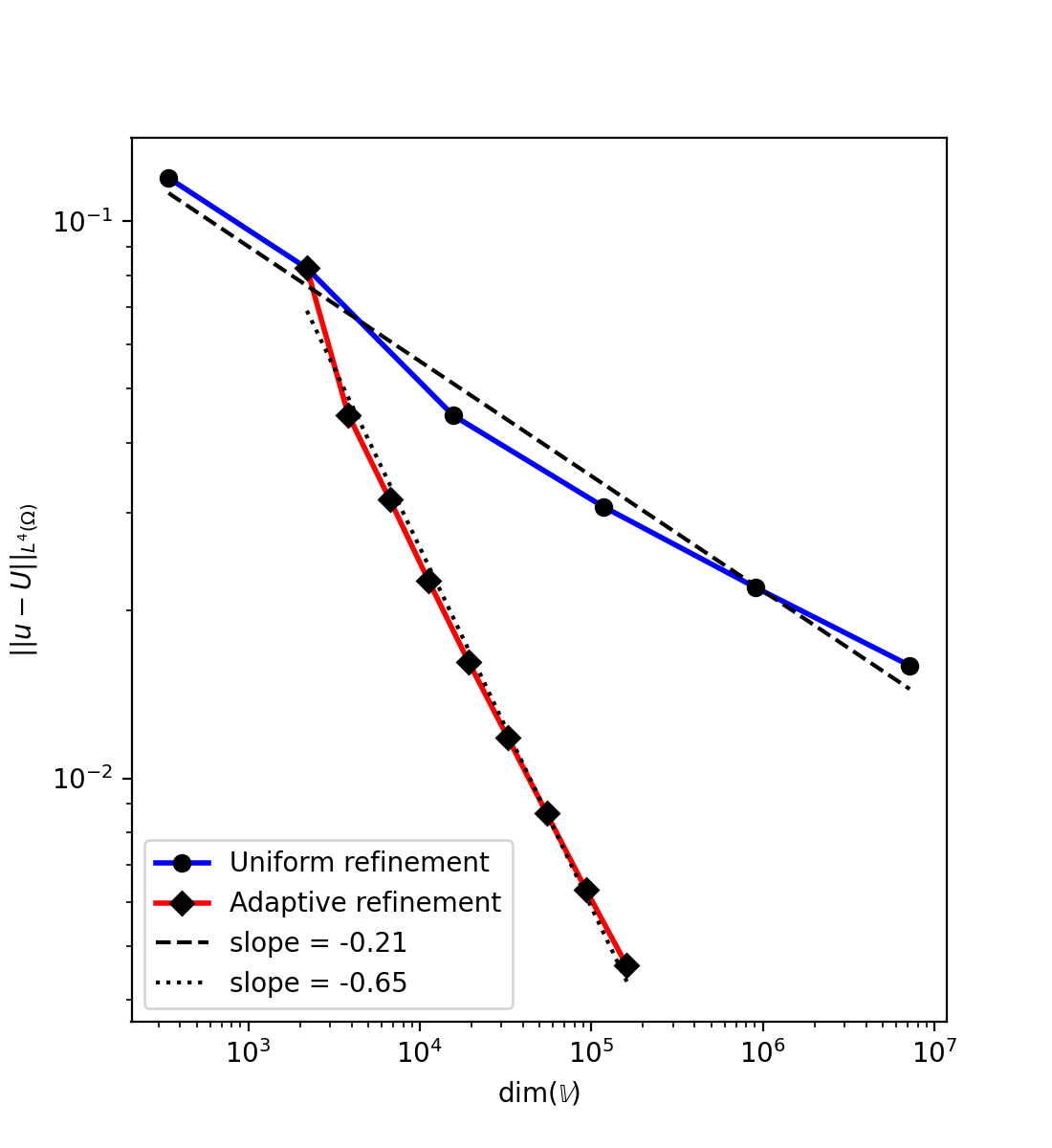}
    \caption{Example 3. Expected slope from theory is $- \tfrac 2 3$.}
    \label{eq:adaptive_uniform_2}
  \end{subfigure}
   \caption{Comparison of $\leb{4}(\W)$ error computed against exact solutions for uniform and adaptive meshes for those examples for which an exact solution is available.
   The orders of convergence of the error estimate in each
   case (that is, the slope of the error values) is approximated from the log-scaled data by performing standard least squares linear
   regression.}
   \label{fig:adaptive_beats_uniform}
\end{figure}

\begin{figure}[ht!]
  \centering
  \begin{subfigure}{0.3\textwidth}
    \includegraphics[width=\linewidth]{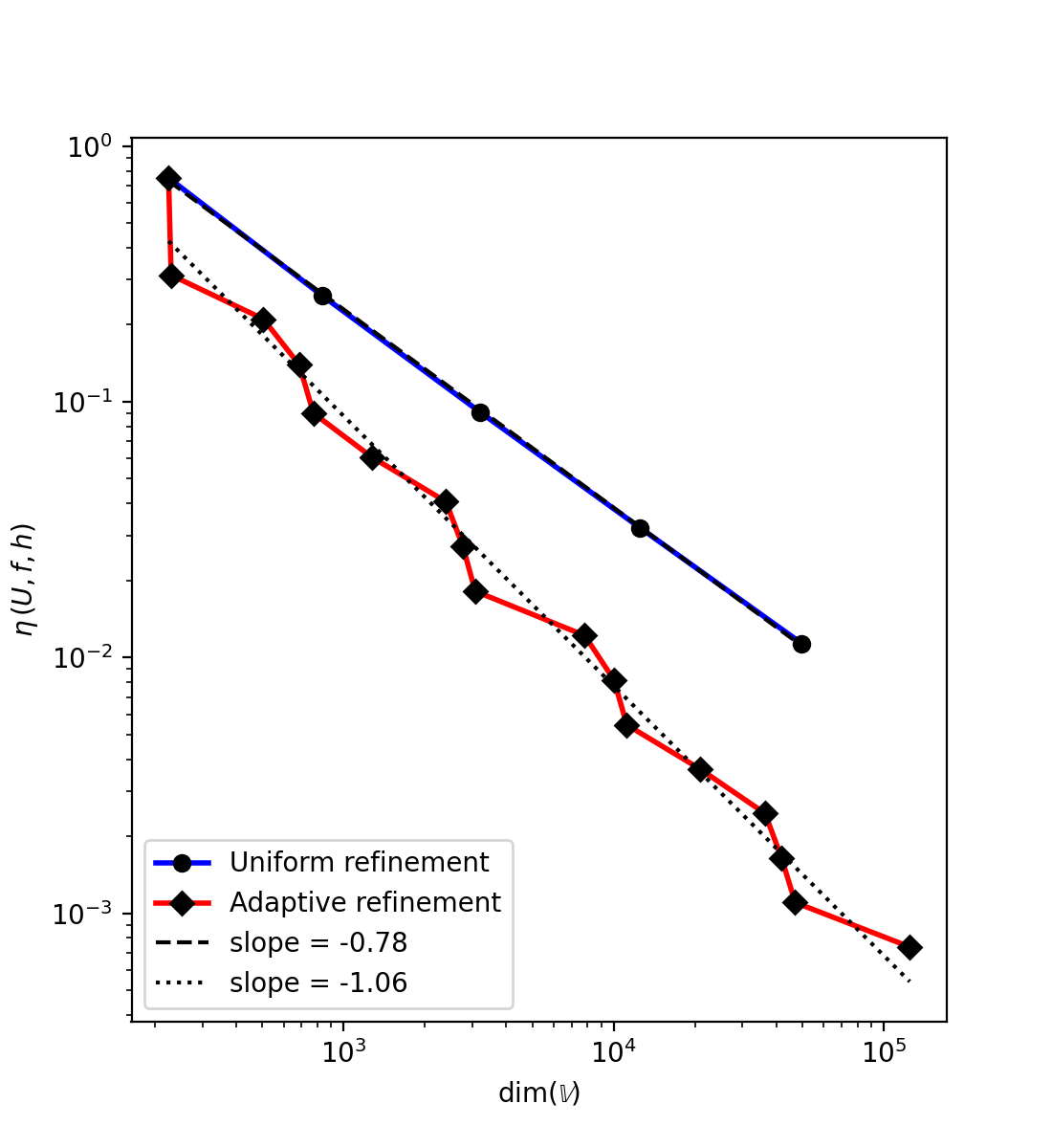}
    \caption{Example 2. Expected slope from theory is $-1$.}
   \end{subfigure}\hfil
 \begin{subfigure}{0.3\textwidth}
    \includegraphics[width=\linewidth]{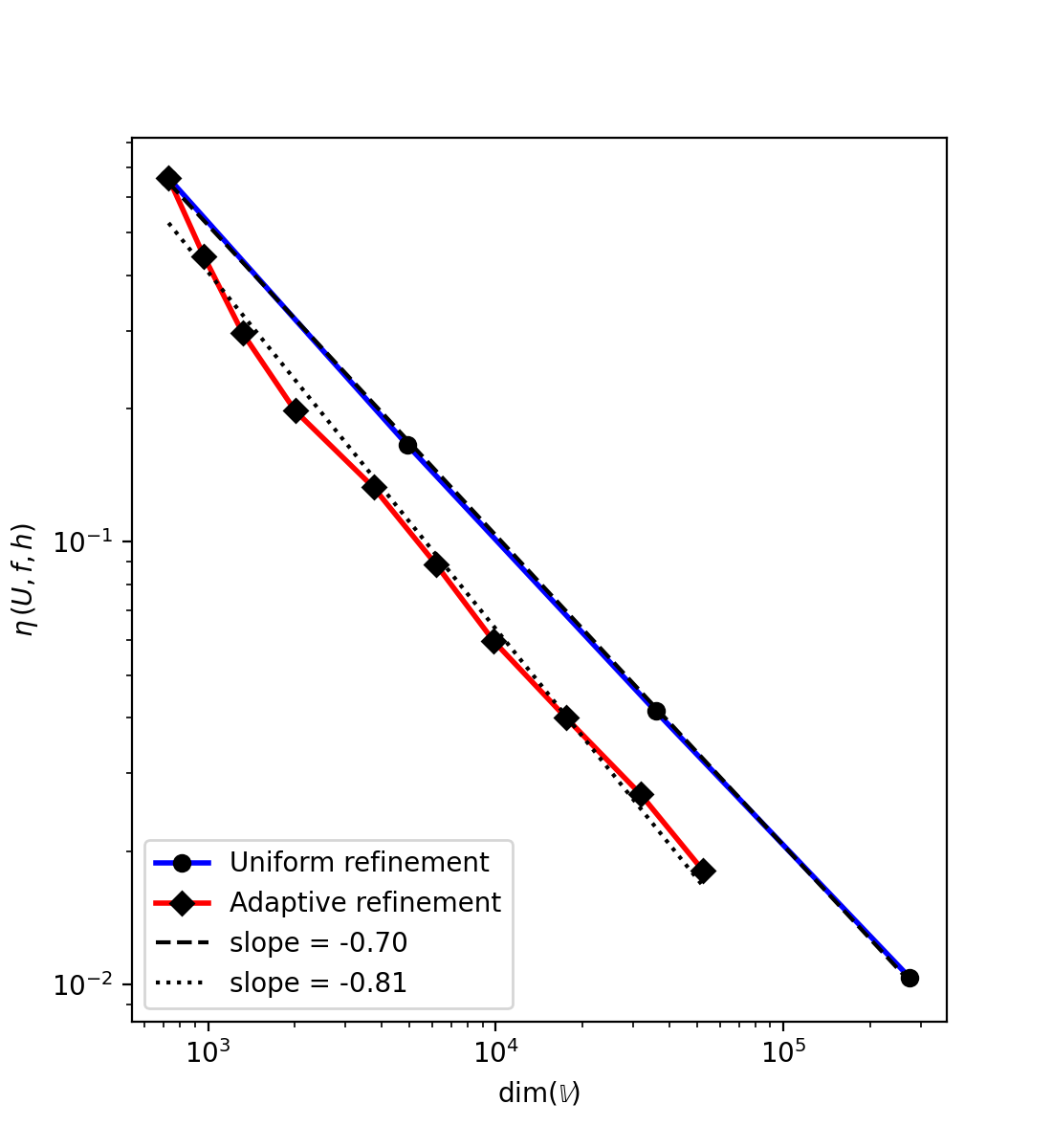}
    \caption{Example 3. Expected slope from theory is $- \tfrac 2 3$.}
    \label{eq:3d_estimate}
    \end{subfigure}
    \caption{ Double logarithmic plot of error estimate
        against number of degrees of freedom under
        uniform and adaptive mesh refinement for examples 2 and 3, that is, the examples that do not meet the necessary assumptions for the theory above to hold. The orders of
        convergence of the error estimate in each
        case is approximated from the log-scaled data by performing standard least squares linear
        regression.}
    \label{suboptimal_convergence_2}
\end{figure}

\subsection{Example 3, \(d = 3\).}\label{sec:l4:ex_4}

Working in three dimensions and in spherical polar coordinates, centred at $(0.5, 0, 0.5)$,  we now consider a test analogous to example 1.
Let $(r, \theta, \phi)$ denote spherical polar coordinates centred at $(0.5,0,0)$.
We note that the function

\begin{equation} \label{eq:3d_sol}
   u=  - 10 \psi(r) r^{3\slash 2} \sin\left( \frac{3}{2} \varphi\right)
    \sin\left(\frac{3}{4}
    \theta\right)
\end{equation}
satisfies appropriate constraints on the boundary and so solves \eqref{eq:diff_eq}, \eqref{eq:constraints} for appropriately defined $f$.
Contours of \eqref{eq:3d_sol} are shown in Figure \ref{fig:3d_slices}.
Again we use D\"orfler marking with refine fraction of $0.8$ and no coarsening.

Adaptive meshes from several stages of the adaptive algorithm are displayed in \ref{fig:3d_meshes} (note that these are slices of a three-dimensional mesh).
We observe that mesh resolution is greater close to the boundary where the Signorini constraints are active (in this case the plane defined by $y=0$) and provides good resolution of the contact set.
Refinement is also present around key features of the solution, as observed for example 1, Figures \ref{eq:sol_1} and \ref{ex_1_mesh1}-\ref{ex_1_mesh4}

\begin{figure}[ht!]
  \centering
    \includegraphics[width=0.25\linewidth]{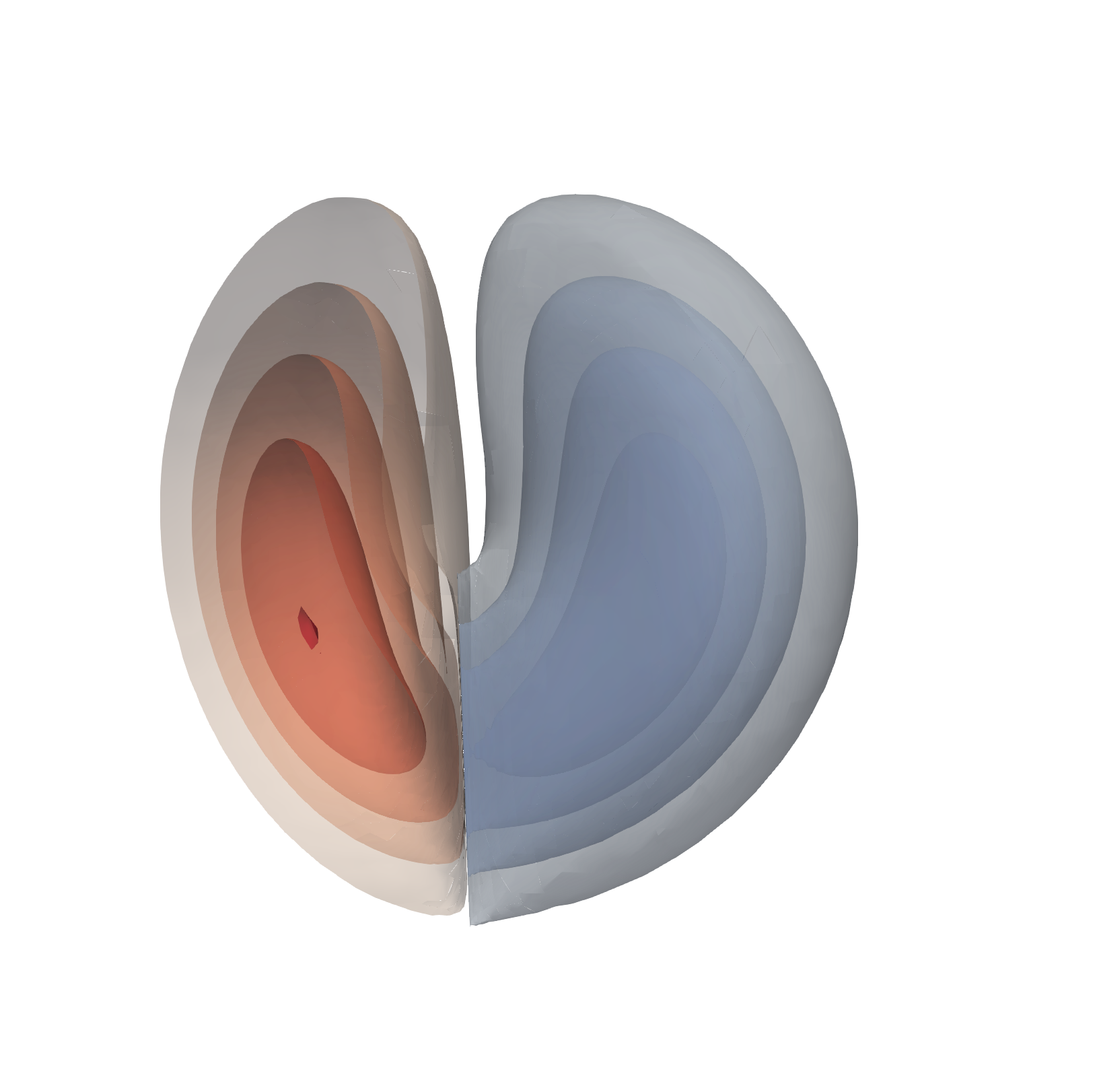}
    \includegraphics[width=0.25\linewidth]{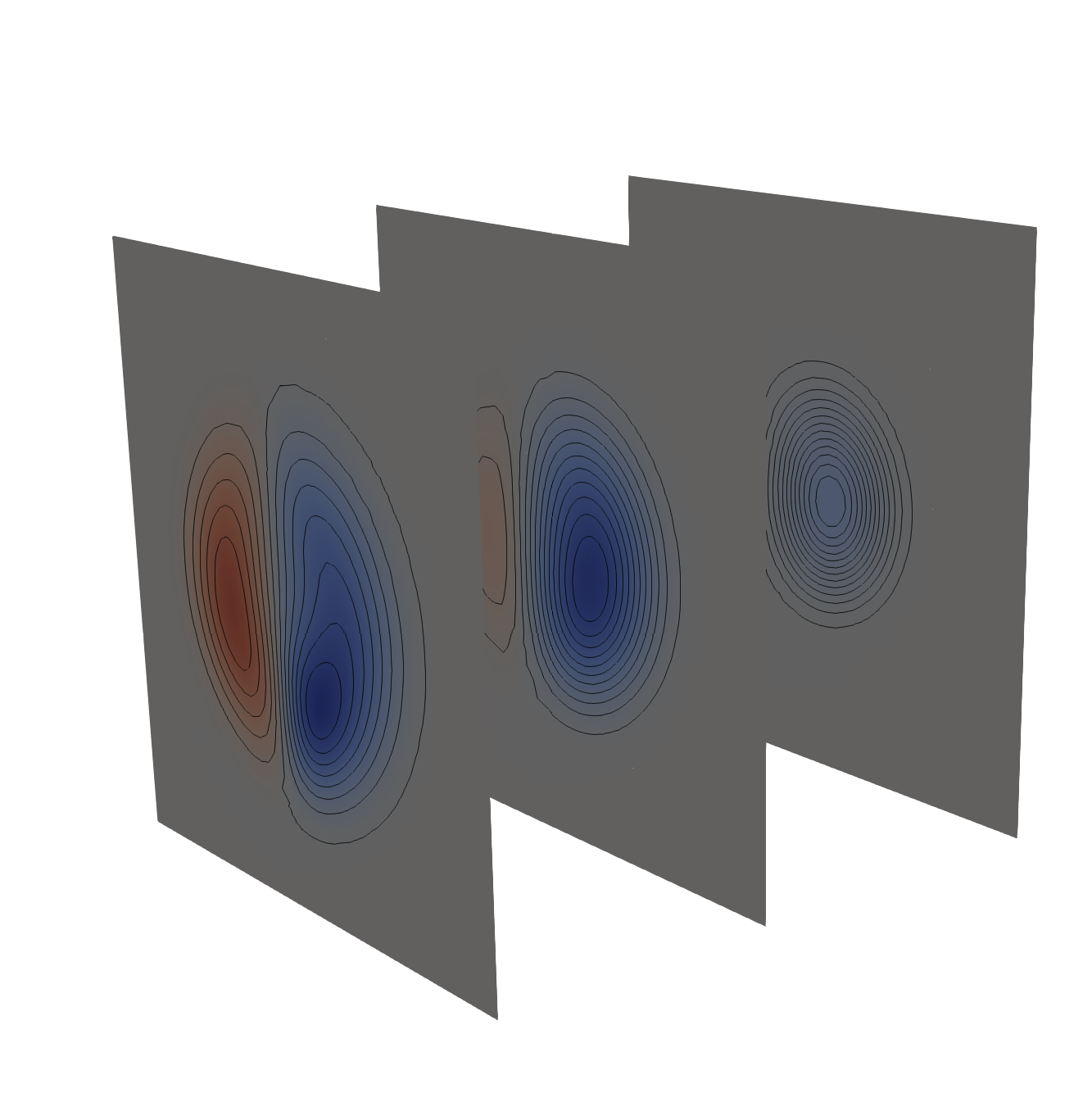} \\
    \includegraphics[width=0.5\linewidth]{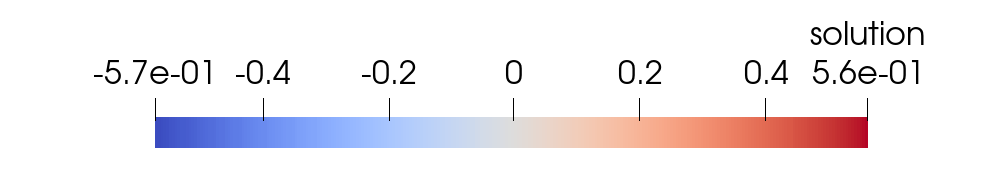}
    \caption{Left: contours of the numerical solution to example 3 \S\ref{sec:l4:ex_4}, Right: slices
    of the numerical solution parallel to planes $y
    = c$ for $c=0.1, 0.2$ and 0.3. Note that for visualisation purposes the right
    hand plot is not to scale; the planes have been translated along the $y$-axis.}
    \label{fig:3d_slices}
\end{figure}

\begin{figure}[ht!]
  \centering
    \includegraphics[width=0.15\linewidth]{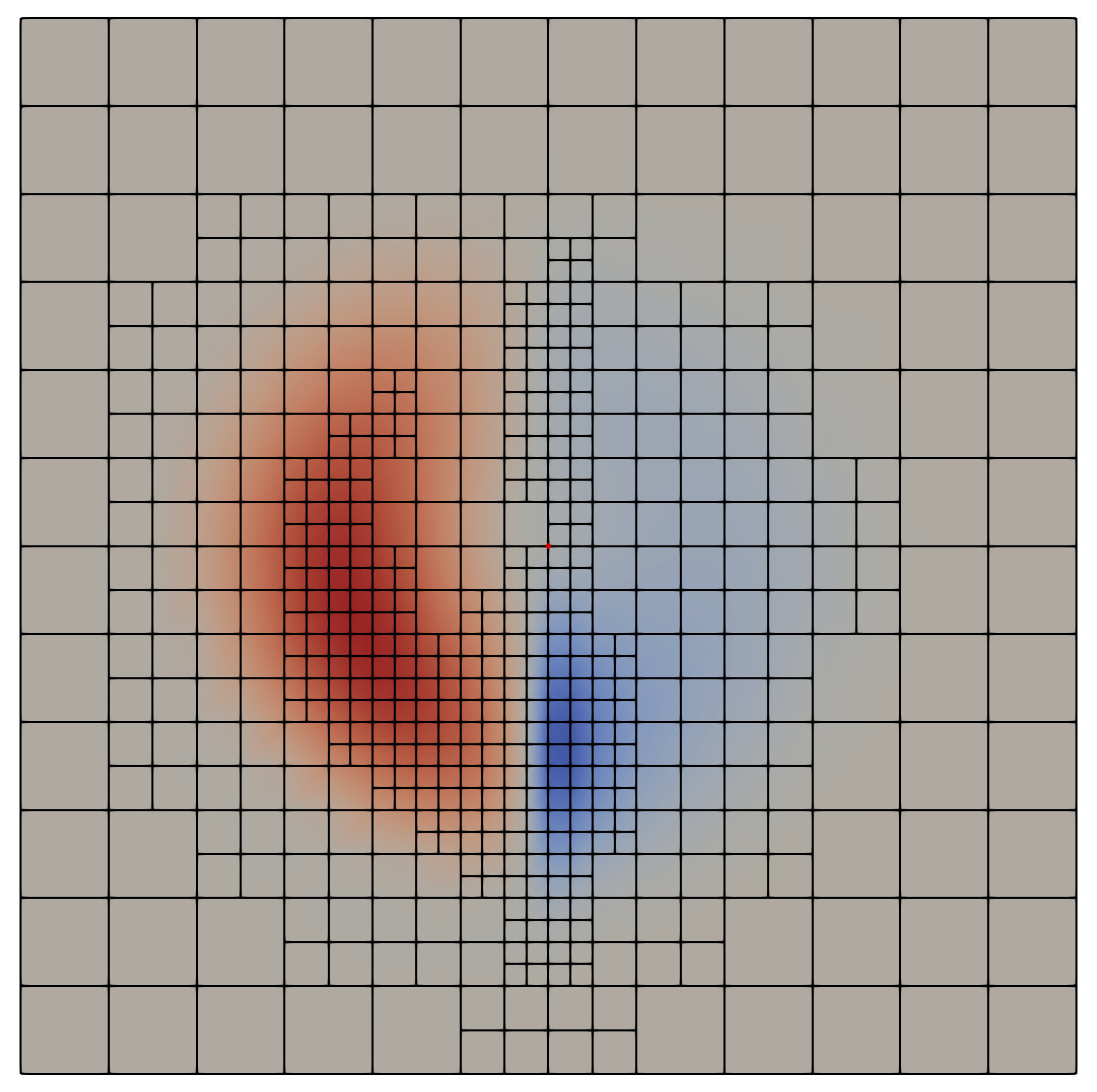}
    \includegraphics[width=0.15\linewidth]{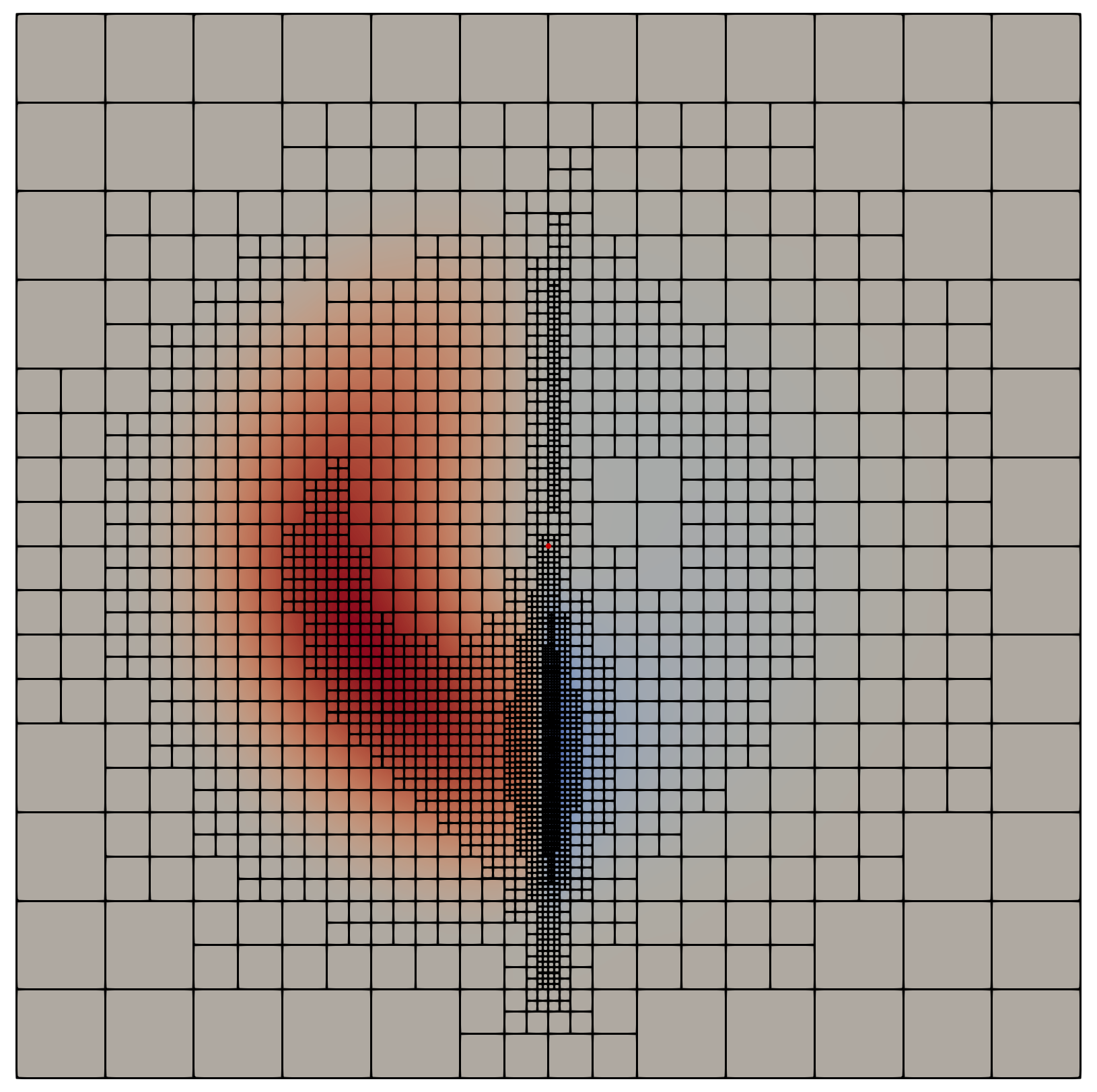}
    \includegraphics[width=0.15\linewidth]{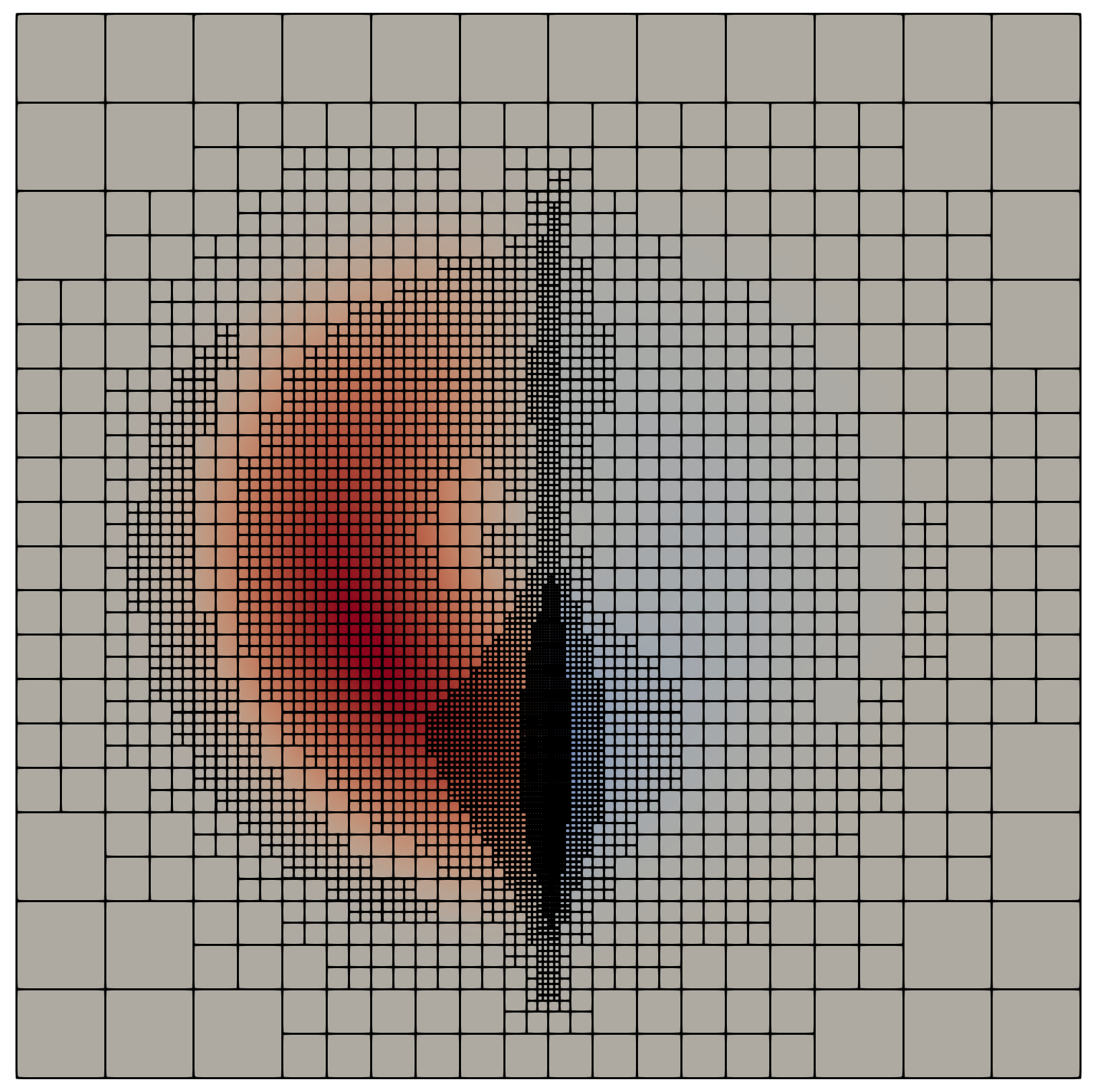} \\
    \includegraphics[width=0.15\linewidth]{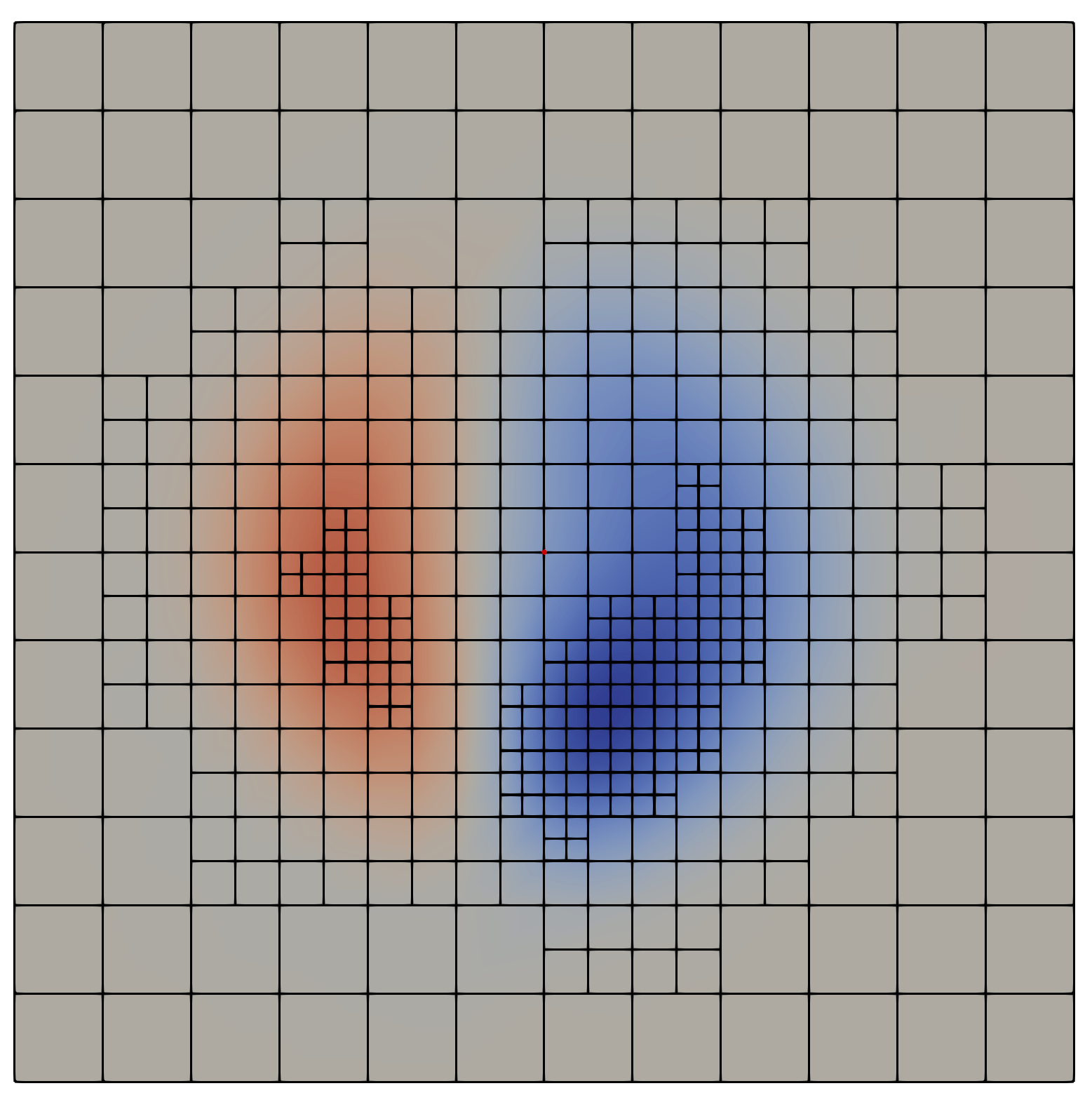}
    \includegraphics[width=0.15\linewidth]{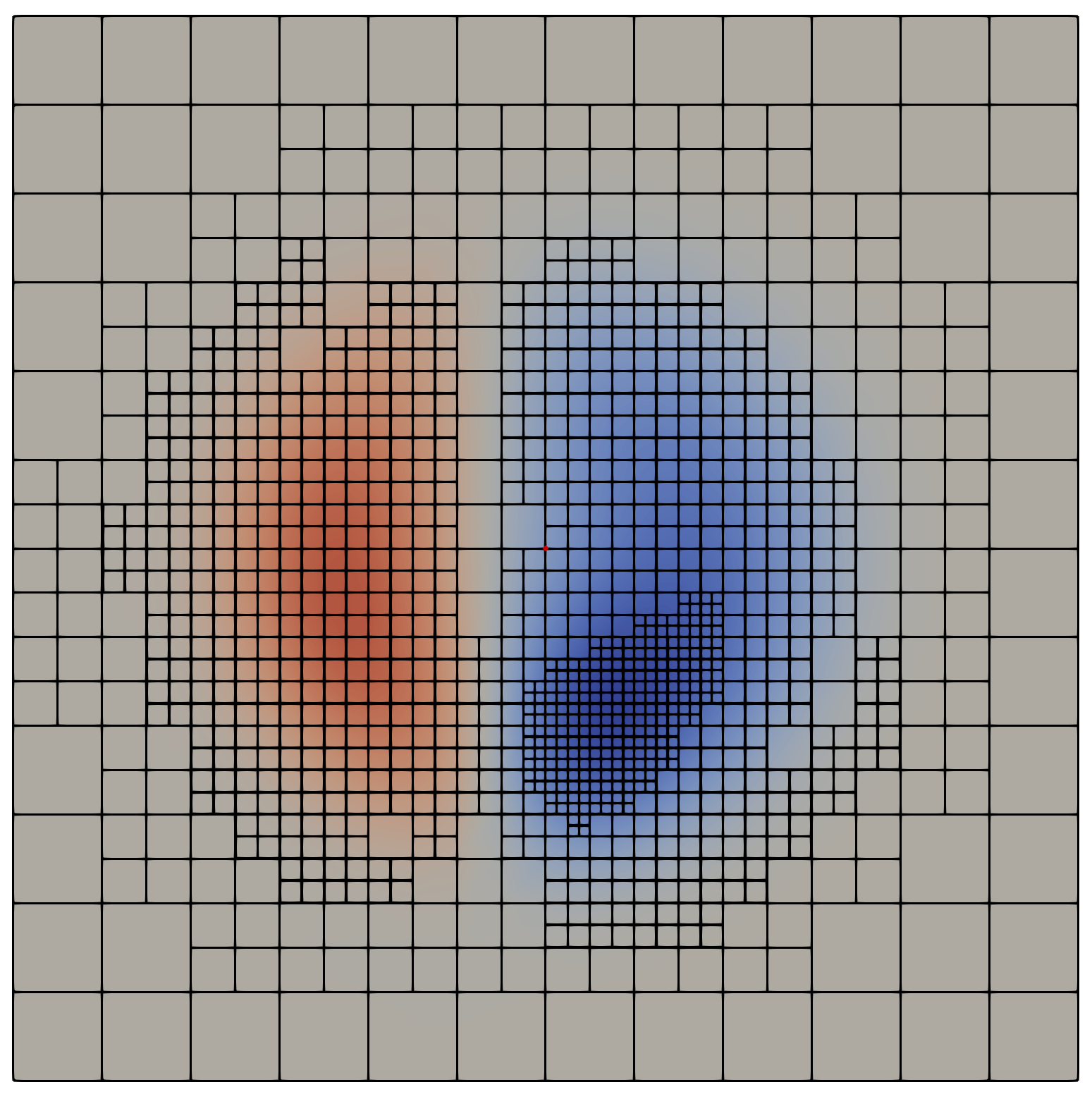}
    \includegraphics[width=0.15\linewidth]{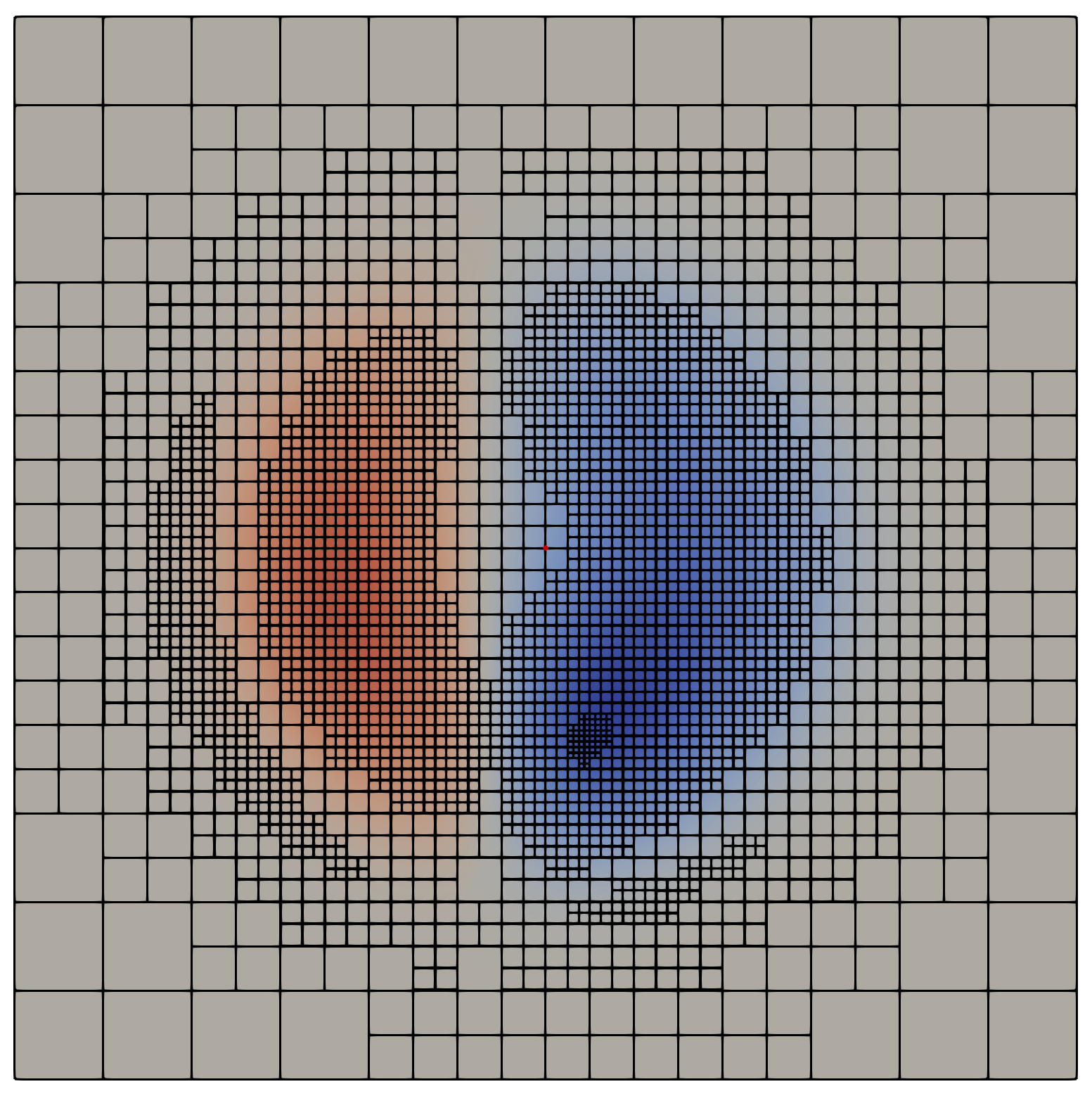} \\
    \includegraphics[width=0.15\linewidth]{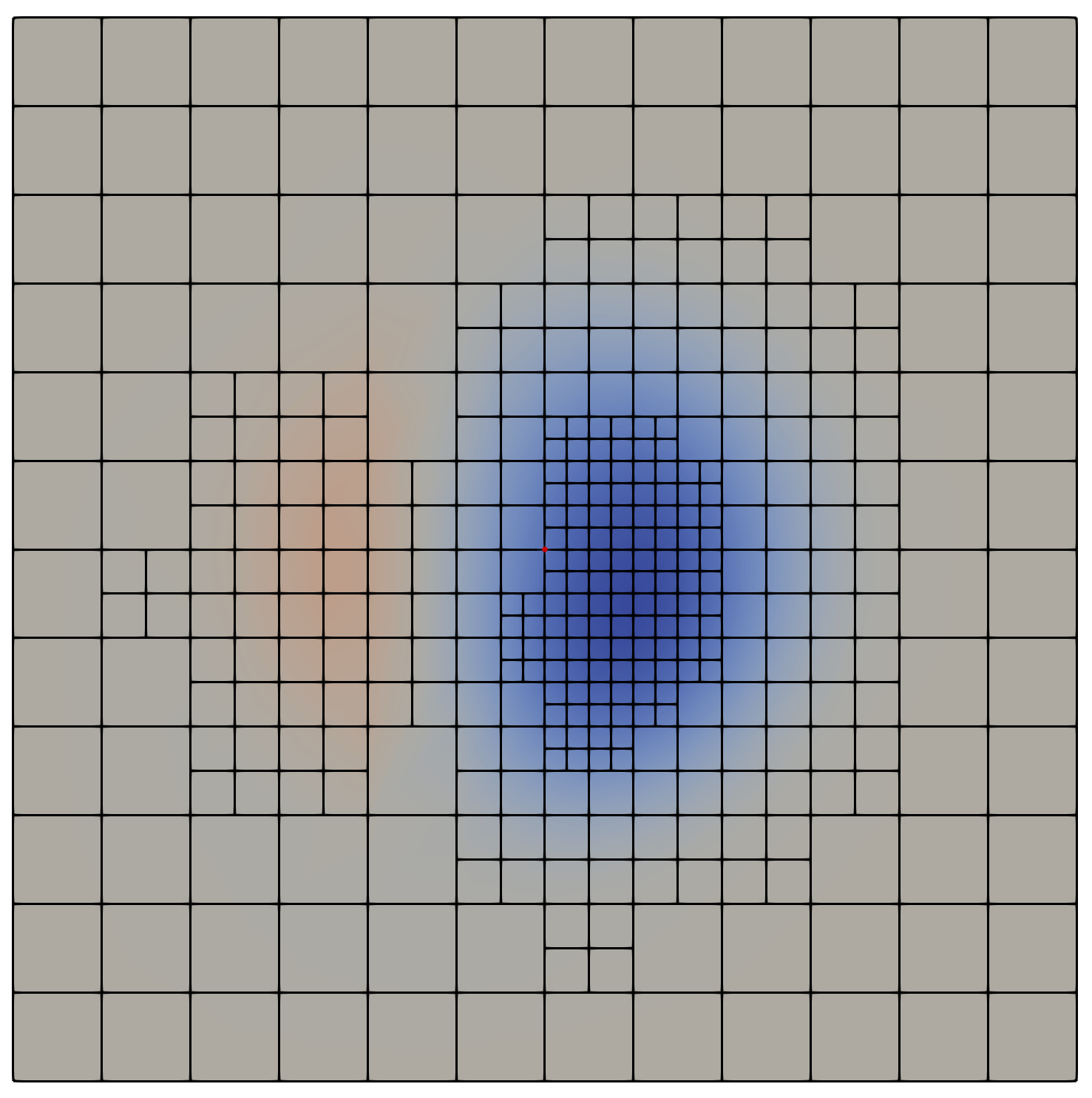}
    \includegraphics[width=0.15\linewidth]{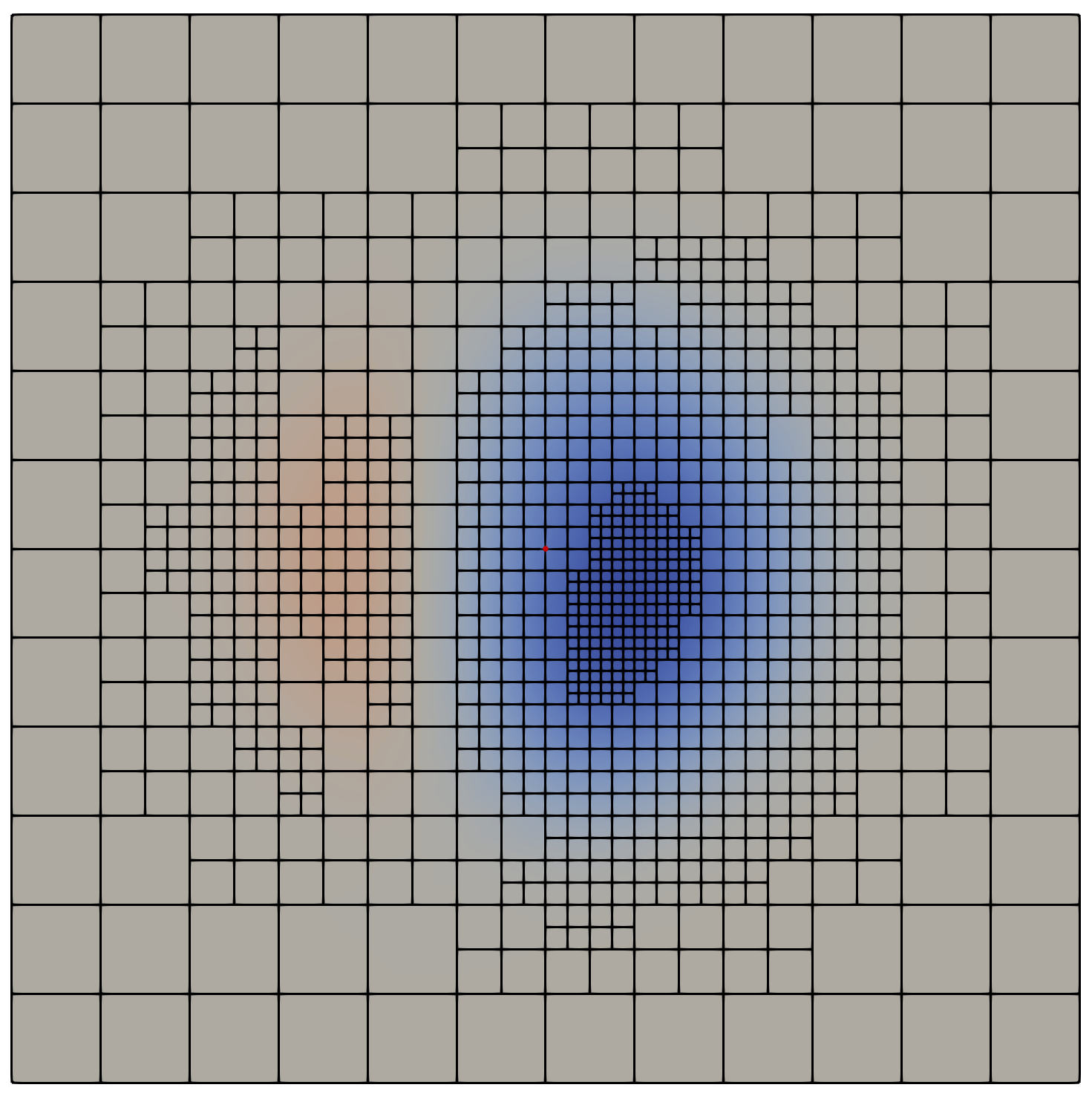}
    \includegraphics[width=0.15\linewidth]{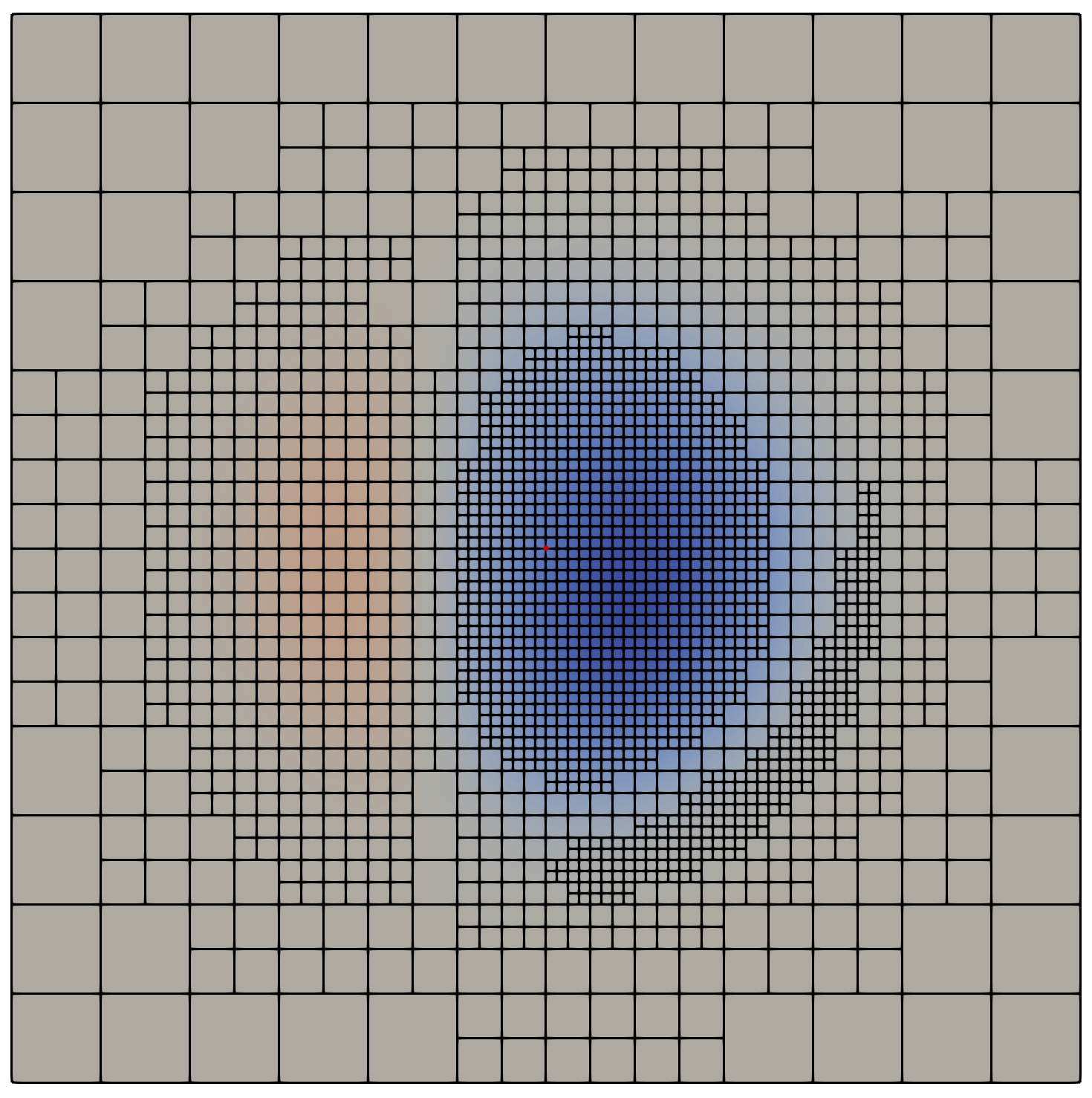}
    \caption{Adaptively generated meshes for example 3. The slices shown are
    the intersections of the mesh with the planes $ y = 0.01$ (top row), $y =
    0.1$, (middle row) and $y = 0.2$ (bottom row) and show the mesh after
    refinement cycles 2 (left hand column), 5 (middle column) and 8 (right
    hand column). The mesh is heavily refined along the $z$ axis due to both
    active boundary constraints and gradient singularities in the solution.}
    \label{fig:3d_meshes}
\end{figure}

\section{Conclusions \& discussion}

In this article, rigorous error estimates were derived for
the scalar Signorini problem in non-energy norms, specifically $\leb
p(\W)$ with $p\in(4,\infty)$, using a duality argument.  The a
posteriori error estimates were benchmarked against known exact
solutions and shown to have the same asymptotic rates as the a
priori analysis, and to be sharp to a satisfactory degree 
(overestimation factor of approximately 20).  Adaptive mesh
refinement was shown to decrease error for a given number of degrees
of freedom compared to uniform meshes, particularly for a
challenging three-dimensional problem.  Furthermore, since
asymptotic convergence is suboptimal in $\leb p(\W)$ for large $p$
we demonstrated that adaptive schemes are able to regain optimal
rates in terms of the number of degrees of freedom.

The error estimate was tested on more challenging test
cases, both in situations where the assumptions of the main theorem
were satisfied, and in a case where they were not (three spatial
dimensions, re-entrant corners).  In the three-dimensional case
where the theory does not hold (see Section
\ref{sec:l4:ex_4} and Figure
\ref{eq:3d_estimate}), we remark that the error estimate is too
optimistic in the sense that the error estimate decreases at the
expected rate under uniform refinement, whereas the error does not
(see Figure \ref{eq:adaptive_uniform_2}).  This means that although
adaptivity did produce superior error reduction in this case, it can
only be used as an indicator, not an estimator.  

This work presents many avenues for further investigation
which include the analysis in higher spatial dimensions, as already
indicated. Further work is underway on nontrivial obstacles as well
as the generalisation to higher polynomial degrees. While we do not
believe the bilateral interpolation result can be true globally,
which is critical to the analysis, we note the recent work
\cite{barrenechea2023nodally} where the authors give an alternative
nodal interpretation that may shed some light on the quest for
higher order methods. 

\section*{Acknowledgements}
B.A. was supported through a PhD scholarship awarded by the EPSRC
Centre for Doctoral Training in the Mathematics of Planet Earth at
Imperial College London and the University of Reading EP\slash
L016613\slash 1. T.P. is grateful for partial support from the EPSRC
programme grant EP\slash W026899\slash 1. Both authors were supported
by the Leverhulme RPG-2021-238.

\begingroup
\setstretch{0.8}
\setlength\bibitemsep{0pt}
\printbibliography
\endgroup


\end{document}